\documentclass[a4paper,12pt
]{article}
\usepackage{amsmath,amssymb,amsthm}
\usepackage{mathtools}
\numberwithin{equation}{section}
\usepackage[margin=20mm]{geometry}
\newtheorem{thm}{Theorem}[section]
\newtheorem{prop}[thm]{Proposition}

\newtheorem{rem}[thm]{Remark}
\newtheorem{lem}[thm]{Lemma}

\def\R{{\mathbb R}}

\def\d{{\rm d}}

\makeatletter

\title{Lyapunov exponents and growth indices for fractional stochastic heat equations with space-time L\'evy white noise}
\author{Yuichi Shiozawa\thanks{Department of Mathematical Sciences, Faculty of Science and Engineering,
Doshisha University,
1-3, Tatara Miyakodani, Kyotanabe, Kyoto, 610-0394,
Japan; \texttt{yshiozawa@mail.doshisha.ac.jp}}\qquad
Jian Wang\thanks{School of Mathematics and Statistics \& Key Laboratory of Analytical Mathematics and Applications (Ministry of Education) \& Fujian Provincial Key Laboratory of Statistics and Artificial Intelligence,
Fujian Normal University, Fuzhou, 350007, P.R. China; \texttt{jianwang@fjnu.edu.cn}}}

\begin{document}
\maketitle
\begin{abstract}
We consider fractional stochastic heat equations with space-time L\'evy white noise of the
form
$$\frac{\partial X}{\partial t}(t,x)={\cal L}_{\alpha}X(t,x)+\sigma(X(t,x))\dot{\Lambda}(t,x).$$
Here, the principal part ${\cal L}_{\alpha}=-(-\Delta)^{\alpha/2}$ is the $d$-dimensional fractional Laplacian with $\alpha\in (0,2)$, the noise term $\dot{\Lambda}(t,x)$ denotes
the
space-time L\'evy white noise, and
the function
$\sigma:
\R\mapsto \R$ is Lipschitz continuous. Under suitable assumptions, we obtain bounds for the Lyapunov exponents and the growth indices of exponential type  on $p$th moments of the mild solutions, which are connected with the weakly intermittency properties and the characterizations of the high peaks propagate away from the origin.
Unlike the case of
the
Gaussian noise, the proofs heavily depend on the heavy
tail property
of heat kernel estimates for the fractional Laplacian.
The results complement these in \cite{CD15-1,CK19} for fractional stochastic heat equations driven by space-time white noise and  stochastic heat equations with L\'evy noise, respectively.
\end{abstract}

\medskip
\noindent
{\bf AMS 2020 
Mathematics subject classification}:
60H15,
60J35, 60J76 

\medskip\noindent
{\bf Keywords and phrases}: fractional stochastic heat equation with space-time L\'evy white noise; the Lyapunov exponent; growth index; the fractional Laplacian

\allowdisplaybreaks

\section{Introduction}
Stochastic heat equations with space-time Gaussian white noise are
one of the important classes of stochastic partial differential equations (see, e.g., \cite{K14, W86} for basic results),
and this research area has been active and developing until now (see \cite{H19} and references therein).
One direction of the developments is to replace the principal part (the Laplace operator) by more general Markov operators
such as the fractional Laplace operator,
which is a non-local operator generating a symmetric stable L\'evy process.
By the interaction between the non-locality of the heat diffusion
and the space-time correlation associated with the Gaussian white noise,
we can observe new phenomena for the solution
such as the intermittency and the multifractal nature
(see, e.g., \cite{CD15-1, CK12, FK12, FLO17, K19} to name a few).

Another direction is to replace the Gaussian white noise by the L\'evy white noise
(\cite{BCL23, BL21, BL22, C17-1,C17-2,CDH19, CK19,CK20, CK22, CK23, S98}).
Since this model has the long-range space-time correlation,
the solution behaves differently from the Gaussian white noise case.
For instance, the higher moments of the solution do not exist in general
(see \cite[Theorem 2.4]{CD15} and \cite[Theorem 3.1]{CK19}).
Moreover, technical difficulties arise and so we could not follow the arguments for the Gaussian white noise case.
As for the moment of the solution, it is difficult to get the lower bound
because of the fact that the Poisson stochastic integrals do not enjoy It\^o's isometry
and the maximal inequality of lower order.
To overcome this difficulty, Chong-Kevei \cite{CK19} established a new and novel inequality
on the lower bound of the Poisson stochastic integrals
(see \cite[Summary of results in p.\ 1914 and Lemma 3.4]{CK19},
or see Lemma \ref{lem:PG-ineq} in
Appendix).
In fact, this inequality is utilized for the weak intermittency and linear intermittency front of the solution
(\cite[Theorems 3.5 and 3.6]{CK19}).

In this paper, we are concerned with the fractional stochastic heat equation with space-time L\'evy white noise (see \eqref{eq:fshe} below).
More precisely, we would like to understand the phenomena caused by the interaction
between the non-locality of the heat diffusion and the long-range space-time correlation.
We here focus on the weak intermittency and exponential intermittency front of the solution,
as well as the condition for the existence and uniqueness of the solution.

\subsection{Setting}
In the following, we first formulate the model of fractional stochastic heat equations with L\'evy white noise.
For $\alpha\in (0,2)$, let ${\cal L}_{\alpha}=-(-\Delta)^{\alpha/2}$ be the $d$-dimensional fractional Laplacian.
Let $q_t(x)$ be the heat kernel associated with ${\cal L}_{\alpha}$.
Then there exists a smooth and strictly decreasing function $g:[0,\infty)\to (0,\infty)$
such that
\begin{equation}\label{eq:heat}
q_t(x)=\frac{1}{t^{d/\alpha}}g\left(\frac{|x|}{t^{1/\alpha}}\right), \quad (t,x)\in (0,\infty)\times {\mathbb R}^d,
\end{equation} where
\begin{equation}\label{eq:g-asymp}
g(r)\sim \frac{c_{d,\alpha}}{r^{d+\alpha}}, \quad r\rightarrow\infty
\end{equation}
for some constant $c_{d,\alpha}>0$ (see \cite[Theorem 2.1]{BG60} and \cite[Proof of Lemma 5]{BJ07}).
In particular, there exist positive constants $C_1$ and $C_2$ such that
\begin{equation}\label{eq:g-decay}
C_1\left(\frac{1}{t^{d/\alpha}}\wedge \frac{t}{|x|^{d+\alpha}}\right)
\le q_t(x)
\le C_2\left(\frac{1}{t^{d/\alpha}}\wedge \frac{t}{|x|^{d+\alpha}}\right),
\quad (t,x)\in (0,\infty) \times \R^d.
\end{equation}

For a measurable set $A \subset [0,\infty)\times {\mathbb R}^d$,
let $|A|=\int_A \d t\,\d x$.
Let ${\cal A}$ be the totality of measurable sets $A \subset [0,\infty)\times {\mathbb R}^d$
so that $|A|<\infty$.
Let $W=W(\d t, \d x)$ be the Gaussian space-time white noise on $[0,\infty)\times {\mathbb R^d}$
with intensity measure $\d t\,\d x$.
Namely, $W$ is a random set function on ${\cal A}$ such that
the family $\{W(A) : A\in {\cal A}\}$
forms a mean-zero Gaussian process with covariance
\[
E[W(A)W(B)]=|A\cap B|, \quad A, B\in {\cal A}.
\]
Let $\lambda$ be a L\'evy measure on ${\mathbb R}$. Namely, $\lambda$ is a positive Borel measure on ${\mathbb R}$ such that
$\int_{{\mathbb R}}(1\wedge  z ^2)\,\lambda(\d z)<\infty$.
Let $\nu(\d t, \d x, \d z)=\d t \, \d x \,\lambda(\d z)$ be
a direct product measure on $[0,\infty)\times {\mathbb R}^d\times {\mathbb R}$,
and let $\mu$ be a Poisson random measure on $[0,\infty)\times {\mathbb R}^d\times {\mathbb R}$
with intensity measure $\nu$.
For $\rho \ge 0$, define the measure $\Lambda=\Lambda(\d t, \d x)$ on $[0,\infty)\times {\mathbb R}^d$ by
$$
\Lambda(\d t, \d x)
=\rho W(\d t, \d x)+\int_{0<|z|<1} z(\mu-\nu)(\d t,  \d x,  \d z)
+\int_{|z|\ge 1} z\,\mu(\d t, \d x, \d z).
$$
This formulation is a random measure counterpart to the L\'evy-It\^o decomposition of L\'evy processes.
If we assume in addition that $\int_{|z|\ge 1}|z|\,\lambda(\d z)<\infty$, then
\begin{equation}\label{eq:lambda-def}
\begin{split}
\Lambda(\d t, \d x)
&=\rho W(\d t, \d x) + \int_{|z|\ge 1} z\,\nu(\d t, \d x, \d z)
+\int_{{\mathbb R}} z(\mu-\nu)(\d t,  \d x,  \d z) \\
&=\rho W(\d t, \d x) +\left(\int_{|z|\ge 1} z\,\lambda(\d z)\right)\, \d t \, \d x
+\int_{{\mathbb R}} z(\mu-\nu)(\d t,  \d x,  \d z) \\
& = : \rho W(\d t, \d x)+b\, \d t \, \d x+\int_{{\mathbb R}} z(\mu-\nu)(\d t,  \d x,  \d z)
\end{split}
\end{equation}
with $b=\int_{|z|\ge 1} z\,\lambda(\d z)$.
\emph{Throughout the paper, we always assume that
$\lambda\ne 0$ and
\begin{equation}\label{eq:a-moment}
\int_{\R} |z|^p\,\lambda(\d z)<\infty \quad \text{for some $p\in [1,1+\alpha/d)$.}
\end{equation}}
Since this yields $\int_{|z|\ge 1}|z|\,\lambda(\d z)<\infty$, $\Lambda(\d t, \d x)$ has the form \eqref{eq:lambda-def}.

\emph{Let $u_0$ be a bounded Borel function on ${\mathbb R}^d$,
and let $\sigma$ be a Lipschitz continuous function on $\R$;
that is, there exists a constant $L_\sigma>0$ such that
\begin{equation}\label{eq:lip}
|\sigma(x)-\sigma(y)|\le L_\sigma|x-y|, \quad x,y\in {\mathbb R}.
\end{equation}}
We consider the following fractional stochastic heat equation with L\'evy white noise:
\begin{equation}\label{eq:fshe}
\left\{
\begin{aligned}
\frac{\partial X}{\partial t}(t,x)&={\cal L}_{\alpha}X(t,x)+\sigma(X(t,x))\dot{\Lambda}(t,x), & (t,x)\in (0,\infty)\times {\mathbb R}^d,\\
X(0,x)&=u_0(x), & x\in {\mathbb R}^d.
\end{aligned}
\right.
\end{equation}
We say that a measurable function $X=X(t,x): \Omega \times [0,\infty) \times {\mathbb R}^d \to {\mathbb R}$
is a mild solution to the equation \eqref{eq:fshe}, if it is predictable and
satisfies the next equation $P$-a.s.\ for any $(t,x)\in (0,\infty)\times \R^d$:
\begin{equation}\label{eq:mild}
\begin{split}
X(t,x)
&=\int_{{\mathbb R}^d}q_t(x-y)u_0(y)\,{\rm d}y+(q*\sigma(X))(t,x) \\
&=\int_{{\mathbb R}^d}q_t(x-y)u_0(y)\,{\rm d}y+\Phi_1(X)(t,x)+\Phi_2(X)(t,x)+\Phi_3(X)(t,x),
\end{split}
\end{equation}
where
\begin{align*}
(q*\sigma(X))(t,x)&=\int_{[0,t)\times {\mathbb R}^d}q_{t-s}(x-y)\sigma(X(s,y))\,\Lambda(\d s, \d y), \\
\Phi_1(X)(t,x)&=\rho \int_{[0,t)\times {\mathbb R}^d} q_{t-s}(x-y)\sigma(X(s,y))\,W(\d s, \d y), \\
\Phi_2(X)(t,x)&=b\int_{[0,t)\times {\mathbb R}^d}q_{t-s}(x-y)\sigma(X(s,y))\,\d s \d y,\\
\Phi_3(X)(t,x)&=\int_{[0,t)\times {\mathbb R}^d\times {\mathbb R}}q_{t-s}(x-y)\sigma(X(s,y))z\, (\mu-\nu)(\d s, \d y, \d z).
\end{align*}
By following the observation in \cite[p.\ 30, 3.5]{K14},
\emph{for $d\ge \alpha$, we take $\rho=0$ so that the Gaussian white noise part disappears}.

\subsection{Lyapunov exponents and growth indices}
We next introduce the notations for the Lyapunov exponents and the growth indices of exponential type,
which describe the growth order of the moment of the solution to the equation \eqref{eq:fshe}.
Let $X=X(t,x)$ be a mild solution to \eqref{eq:fshe}.
For $p\ge 0$, define the upper and lower Lyapunov
exponents
of order $p$ by
\[
\overline{\gamma}(p)
=\limsup_{t\rightarrow\infty}\frac{1}{t}\sup_{x\in {\mathbb R}^d} \log E[|X(t,x)|^p]
\]
and
\[
\underline{\gamma}(p)=\liminf_{t\rightarrow\infty}\frac{1}{t}\inf_{x\in \R^d}\log E[|X(t,x)|^p].
\]
For $p\ge 0$, we also define the growth indices of exponential type:
\[
\overline{\lambda}(p)
=\inf\left\{\eta>0 : \limsup_{t\rightarrow\infty}\frac{1}{t}\sup_{|x|\ge e^{\eta t}} \log E[|X(t,x)|^p] <0\right\}
\]
and
\[
\underline{\lambda}(p)
=\sup\left\{\eta>0 : \limsup_{t\rightarrow\infty}\frac{1}{t}\sup_{|x|\ge e^{\eta t}} \log E[|X(t,x)|^p]>0\right\}.
\]
By definition, $\overline{\gamma}(p)\ge \underline{\gamma}(p)$ and
$\overline{\lambda}(p)\ge \underline{\lambda}(p)$ for any $p\ge 0$.

The upper and lower Lyapunov
exponents
of order $p$, $\overline{\gamma}(p)$ and $\underline{\gamma}(p)$, are related to the phenomenon of the weakly intermittency of the mild solution $X$.
We refer the reader to \cite{FK09} and references therein for the mathematical physics background on the (weak) intermittency.
Roughly speaking, the intermittency property implies that the sample paths of $X(t,x)$ exhibit high peaks separated by large valleys.
Conus and Khoshnevisan \cite[(1.6) and (1.7)]{CK12} first introduced
the growth
indices of linear type
(that is, $|x|\ge e^{\eta t}$ in the definitions of $\overline{\lambda}(p)$ and $\underline{\lambda}(p)$ is replaced by $|x|\ge \alpha t$)
for a stochastic heat equation with Gaussian white noise such that
the principal term is the L\'evy generator with
L\'evy measure satisfying the finite second moment condition.
This condition yields the light tail behavior of the heat propagation
speed.
On the other hand, Chen and Dalang \cite[(1.4) and (1.5)]{CD15-1} introduced
the growth
indices of exponential type
by taking into consideration the heavy tail behavior of the heat propagation.
Our definitions of  $\overline{\lambda}(p)$ and $\underline{\lambda}(p)$ are essentially the same
as \cite[(1.4) and (1.5)]{CD15-1}.
It is clear that
the growth indices of exponential type
characterize
the speed
in which high peaks propagate away from the origin.

The main purpose of this paper is to study upper bounds and lower bounds of the Lyapunov exponents and the growth indices of exponential type
for the mild solutions to fractional stochastic heat equations with L\'evy white noise defined by \eqref{eq:fshe}.
To illustrate the contribution of our paper, here we are concerned on the following especial case that $\alpha>d=1$.

\begin{thm}\label{thm1}
Let $\alpha>d=1$.
Suppose that
the measure $\lambda$
is symmetric and satisfies \eqref{eq:a-moment}.
Let $X=X(t,x)$ be the mild solution to \eqref{eq:fshe}
such that the initial function $u_0(x)$ is bounded on ${\mathbb R} $,
and the coefficient $\sigma$ is a Lipschitz continuous function on $\R$ such that $\inf_{w\in \R\backslash\{0\}}{|\sigma(w)|}/{|w|}>0$.
Then, the following statements hold.
\begin{itemize}
\item[{\rm(i)}] If $\inf_{x\in \R}u_0(x)>0$, then, for any $p\in(1,1+\alpha)$,
$$0<\underline{\gamma}(p)\le\overline{\gamma}(p)<\infty.$$
\item[{\rm(ii)}] If $\sigma(0)=0$, and $0\le u_0(x)\le C_0(1+|x|)^{-c}$ for some $C_0>0$ and $c\in (0,\alpha)$ such that
$u_0$
is strictly positive on a set of positive Lebesgue measure, then, for any $p\in [2,1+\alpha)$,
$$0<\underline{\lambda}(p)\le \overline{\lambda}(p)<\infty.$$
\end{itemize} \end{thm}

Theorem \ref{thm1}(i) extend and complement \cite[Theorem 2.4 and Theorem 3.5(4)]{CK19} for stochastic heat equations driven by L\'evy space-time white noise,
while Theorem \ref{thm1}(ii) generalizes \cite[Theorem 3.6]{CD15-1} for nonlinear fractional stochastic heat equations in the spatial domain $\R$ driven by space-time white noise.

Below we emphasize, by comparing the existing results in the literature, the differences and the difficulties of the proofs for Theorem \ref{thm1} and its generalized results (see Sections \ref{section3} and \ref{section4} below). This partly 
highlights the novelties of our paper.
\begin{itemize}
\item[{\rm(i)}] To investigate upper bounds for the indices $\overline{\gamma}(p)$ and $ \overline{\lambda}(p)$,
we consider the mild solution to \eqref{eq:fshe} in some suitable norm,
which
in turn yields the existence and the uniqueness of the mild solution to \eqref{eq:fshe}. The approach is inspired by that in \cite[Section 2]{CK19}. However, by \eqref{eq:g-decay}, the heat kernel $q_t(x)$ of the principal part -- the fractional  Laplacian  ${\cal L}_{\alpha}=-(-\Delta)^{\alpha/2}$ -- enjoys the polynomially decay in spatial variable, and so we will adopt the polynomial growth weighted function $(1+|x|)^c$ in the suitable norm.
This is in comparison with the framework in \cite[Section 2]{CK19},
where
the principal part is the Laplacian and so
the exponential growth spatial weighted function $e^{c|x|}$ is used.
Note that, there are also some delicate differences involved in. For example, in our paper the constant $c$ in the weighted function $(1+|x|)^c$ is required to be $c\in [0,\alpha)$, while in \cite{CK19} the constant $c$ in the weighted function $e^{c|x|}$ can be taken to be any $c>0$. Besides, the function $|x|\mapsto e^{c|x|}$ has the additive property, and the function $|x|\to (1+|x|)^c$ only has the multiplicative property (see \eqref{eq:triangle}). This causes the arguments to be a little cumbersome in our paper.

\item[{\rm(ii)}] As mentioned in \cite{CK19}, the main difference between the stochastic heat equations with
L\'evy
noise and with Gaussian noise is that the mild solution $X$ to \eqref{eq:fshe} has no large moments
(see Remarks \ref{rem:h-moment-1} and \ref{rem:h-moment-2}).
To consider the intermittency lower bounds  for $\underline{\gamma}(p)$, we mainly follow the proof of \cite[Theorem 3.5]{CK19} via the renewal inequalities. Indeed, similar to \cite[Theorem 3.5]{CK19}, we can deal with the case for $p\in(1,1+\alpha/d)$ when $d\ge \alpha$; see Proposition \ref{thm:growth-lower}. In particular, we can treat the case $p\in (1,2)$, where the novel inequality
on the lower bound of the Poisson stochastic integrals
(see \cite[Lemma 3.4]{CK19} or Lemma \ref{lem:PG-ineq} in
Appendix)
is fully used.
On the other hand, as pointed out before, the heat kernel $q_t(x)$ of the fractional  Laplacian  ${\cal L}_{\alpha}=-(-\Delta)^{\alpha/2}$ has the polynomially decay, and so in our setting we will consider intermittency fronts $\underline{\lambda}(p)$ and $\overline{\lambda}(p)$ of the exponential type.
Because of
this difference, the approach of \cite[Theorem 1.6]{CK19} does not work.
Instead of this approach, we
establish lower bounds for $\underline{\lambda}(p)$ by means of the proof of \cite[Theorem 3.6(2)]{CD15-1}.
For this,
we need
to
develop a few convolution inequalities and moments inequalities for the heat kernel
$q_t(x)$;
see Section \ref{section6.2} for the details.

\item[{\rm(iii)}] We make some additional comments on Theorem \ref{thm1}(ii).
It is easy to deduce from the suitable norm of the mild solution to \eqref{eq:fshe} to get upper bounds for $\overline{\lambda}(p)$ with all $p\in (1,\alpha/d)$;
see Proposition \ref{prop:exp-upper}.
However, to obtain the lower bounds for $\underline{\lambda}(p)$ when $p\in(1,2)$ (which happens when $d\ge \alpha$ since
$p\in [1,1+\alpha/d)$
is required throughout the paper).
Even
though
a complete result in this direction is still open, we can obtain satisfactory lower bounds for the growth indices of subexponential type for $p\in (1,2)$; see Section \ref{section5}.
In particular, the proof of Proposition \ref{thm:p<2} here is based on the lower bound of the Poisson stochastic integrals
(see \cite[Lemma 3.4]{CK19} or Lemma \ref{lem:PG-ineq}),
the convolution inequalities and moments inequalities for the heat kernel
$q_t(x)$,
as well as a new auxiliary Poisson random measure $\mu_0$ defined by \eqref{eq:mu0}.
\end{itemize}

\ \

The rest of the paper is arranged as follows. The next section is devoted to the existence and
the
uniqueness of the mild solution to \eqref{eq:fshe} in some suitable norm. In Sections \ref{section3} and \ref{section4}, we present bounds for the Lyapunov exponents and the growth indices of exponential type  on $p$th moments of the mild solutions, respectively. In Section \ref{section5}, we consider lower bounds for the growth indices of subexponential type for $p\in (1,2)$. Some auxiliary lemmas are presented, and new convolution inequalities and moments inequalities for the heat kernel
$q_t(x)$ are proved in Appendix.

\section{Existence and uniqueness of mild solutions}
In this section, we establish the existence and
the uniqueness of the mild solution to \eqref{eq:fshe} in some suitable norm,
which is useful to obtain upper bounds for the Lyapunov exponents and the growth indices of exponential type.

Let $(\Omega, {\cal F}, P)$ be a probability space.
Let $p\ge 1$, $\beta>0$ and $c\ge 0$.
For a measurable function
$\Phi=\Phi(t,x): \Omega \times [0,\infty) \times {\mathbb R}^d \to {\mathbb R}$,
define
\begin{equation}\label{eq:norm}
\|\Phi(t,x)\|_p=E [|\Phi(t,x)|^p]^{1/p}, \quad \| \Phi \|_{\beta,c,p}
=\sup_{t\ge 0, \, x\in {\mathbb R}^d}\left\{e^{-\beta t}(1+|x|)^c \|\Phi(t,x)\|_p\right\}.
\end{equation}
We note that for $\alpha=2$,
Chong-Kevei \cite[(2.1)]{CK19} introduced the norm $\|\cdot\|_{\beta, c, p}$
with the exponential growth spatial
weighted function $e^{c|x|}$ instead of $(1+|x|)^c$.
This is due to the fact that the heat kernel of ${\cal L}_2$ is Gaussian and so decays exponentially in spatial variable.
On the other hand, for $\alpha\in (0,2)$,
since the heat kernel $q_t(x)$ decays polynomially in spatial variable (see \eqref{eq:g-decay}),
we use here the polynomial growth weighted function $(1+|x|)^c$.

We are now in a position to discuss the existence and the uniqueness of the mild solution to \eqref{eq:fshe}
in the norm $\|\cdot\|_{\beta,c,p}$ with proper choices of $\beta, c$ and $p$.
The main result of this section is as follows.

\begin{thm}\label{thm:ex-un} Let $p\in [1,1+\alpha/d)$ and $c\in [0,\alpha)$.
Assume in addition that $\rho=0$ if $p<2$. Then, the following statements hold.
\begin{enumerate}
\item[{\rm (1)}]
If $X=X(t,x)$ and $Y=Y(t,x)$ are mild solutions to \eqref{eq:mild},
then $X=Y$.
\item[{\rm (2)}]
Assume that either of the following conditions holds{\rm :}
\begin{enumerate}
\item[{\rm (i)}] $\sigma(0)=0$ and $\sup_{y\in {\mathbb R}^d}\{(1+|y|)^c|u_0(y)|\}<\infty$ for some $c\in [0,\alpha)$.
\item[{\rm (ii)}] $\sigma(0)\ne 0$ and $c=0$, i.e., $\sup_{y\in \mathbb R^d} |u_0(y)|<\infty$.
\end{enumerate}
Then there exists $\beta_0>0$ such that for any $\beta\ge \beta_0$,
there exists a mild solution $X$ to the fractional stochastic heat equation \eqref{eq:fshe} such that
$\|X\|_{\beta,c,p}<\infty$.
\end{enumerate}
\end{thm}

To show Theorem \ref{thm:ex-un}, we present two statements.
Let $p\in [1,1+\alpha/d)$ and $c\in [0,d+\alpha)$ so that $p>d/(d+\alpha-c)$. For any $\beta>0$, define
\begin{equation}\label{eq:i}
\begin{split}
I(\beta,c,p)
=&\frac{p(d+\alpha)}{d(p(d+\alpha)-d)}
\Gamma \left( 1-\frac{d}{\alpha}(p-1)\right)(p\beta)^{(p-1)d/\alpha -1}\\
&+\frac{p(d+\alpha)}{(d+pc)(p(d+\alpha-c)-d)}
\Gamma \left( 1 + \frac{pc}{\alpha} - \frac{d}{\alpha}(p-1)\right)(p\beta)^{(p-1)d/\alpha - pc/\alpha -1}.
\end{split}
\end{equation}
It is obvious that $I(\beta,c,p)$ is decreasing with respect to $\beta$ such that $I(\beta,c,p)\rightarrow 0$ as $\beta\rightarrow\infty$.

\begin{lem}\label{lem:int-bound}
Let $p\in [1,1+\alpha/d)$ and $c\in [0,d+\alpha)$ so that $p>d/(d+\alpha-c)$.
Then there exist positive constants $\kappa_1:=\kappa_1(c,p)$ and $\kappa_2:=\kappa_2(c,p)$ such that
for any $\beta>0$,
\[
\kappa_1 I(\beta,c,p)
\le \int_0^{\infty}\left( \int_{{\mathbb R}^d} \{e^{-\beta t}(1+|x|)^c q_t(x)\}^p \, \d x \right)\, \d t
\le \kappa_2 I(\beta,c,p).
\]
\end{lem}

\begin{proof}
By \eqref{eq:heat},
\begin{align*}
&\int_0^{\infty}\left( \int_{{\mathbb R}^d} \{e^{-\beta t}(1+|x|)^c q_t(x)\}^p \, \d x \right)\, \d t \\
&=\int_0^{\infty}\left( \int_{{\mathbb R}^d} e^{- p \beta t}(1+|x|)^{pc} \frac{1}{t^{pd/\alpha}}g\left(\frac{|x|}{t^{1/\alpha}}\right)^p \, \d x \right)\, \d t \\
&=\int_0^{\infty}\left( \int_{|x| \le t^{1/\alpha}} e^{- p \beta t}(1+|x|)^{pc} \frac{1}{t^{pd/\alpha}}g\left(\frac{|x|}{t^{1/\alpha}}\right)^p \, \d x \right)\, \d t \\
&\quad +\int_0^{\infty}\left( \int_{|x| > t^{1/\alpha}} e^{- p \beta t}(1+|x|)^{pc} \frac{1}{t^{pd/\alpha}}g\left(\frac{|x|}{t^{1/\alpha}}\right)^p \, \d x \right)\, \d t\\
&=:{\rm (I)}+{\rm (II)}.
\end{align*}
Note that $g$ is uniformly positive and bounded on $[0,1]$,
and that there exist positive constants $c_1:=c_1(c,p)$ and $c_2:=c_2(c,p)$ so that
$c_1(1+|x|^{pc})\le (1+|x|)^{pc} \le c_2(1+|x|^{pc})$ for any $x\in \R^d$.
Then,
\begin{align*}
{\rm (I)}
&\asymp \int_0^{\infty}\left( \frac{1}{t^{dp/\alpha}}\int_{|x| \le t^{1/\alpha}} e^{- p \beta t}(1+|x|^{pc})  \, \d x \right) \, \d t\\
&=\frac{\omega_d}{d}\int_0^{\infty}e^{-p\beta t}t^{-(p-1)d/\alpha}\,\d t
+\frac{\omega_d}{pc+d}\int_0^{\infty}e^{-p\beta t} t^{(pc-(p-1)d)/\alpha}\,\d t\\
&=\frac{\omega_d}{d}(p \beta)^{(p-1)d/\alpha-1}\Gamma\left(1-\frac{d}{\alpha}(p-1)\right)
+\frac{\omega_d}{pc+d} (p\beta)^{((p-1)d-pc)/\alpha-1}\Gamma\left(1+\frac{pc-(p-1)d}{\alpha}\right),
\end{align*}
where $\omega_d$ is the surface area of the unit ball in ${\mathbb R}^d$, and in the last equality we used the fact that $p<1+\alpha/d$.
On the other hand, by \eqref{eq:g-asymp}, we have $g(r)\asymp r^{-(d+\alpha)}$ for $r\ge 1$, and so
\begin{align*}
{\rm (II)}
&\asymp \int_0^{\infty}\left( \frac{1}{t^{pd/\alpha}}
\int_{|x| >t^{1/\alpha}} e^{- p \beta t}(1+|x|^{pc}) \frac{t^{p(d+\alpha)/\alpha}}{|x|^{p(d+\alpha)}} \, \d x \right) \, \d t\\
&=\int_0^{\infty}e^{-p\beta t}t^p \left(\int_{|x|>t^{1/\alpha}}(|x|^{-p(d+\alpha)}+|x|^{p(c-d-\alpha)})\, \d x \right)\, \d t \\
&=\frac{\omega_d}{p(d+\alpha)-d} \int_0^{\infty}e^{-p\beta t} t^{-(p-1)d/\alpha}\, \d t
+ \frac{\omega_d}{p(d+\alpha-c)-d} \int_0^{\infty}e^{-p\beta t} t^{(pc-(p-1)d)/\alpha}\, \d t \\
&=\frac{\omega_d}{p(d+\alpha)-d} (p \beta)^{(p-1)d/\alpha-1} \Gamma\left(1-\frac{d}{\alpha}(p-1)\right) \\
&\quad + \frac{\omega_d}{p(d+\alpha-c)-d} (p \beta)^{((p-1)d-pc)/\alpha-1} \Gamma\left(1+\frac{pc-(p-1)d}{\alpha}\right),
\end{align*}
where the second equality follows from the  conditions that $c<d+\alpha$ and $p>d/(d+\alpha-c)$.
Hence, combining all the estimates above, we complete the proof.
\end{proof}

\begin{prop}\label{prop:contract}
Let $p\in [1,1+\alpha/d)$ and $c\in [0,\alpha)$.
Assume in addition that $\rho=0$ if $p<2$.
Then, for any $\beta>0$, there exists a positive constant $c_{\beta,c,p}$ such that
for any measurable function
$\Phi=\Phi(t,x): \Omega \times [0,\infty) \times {\mathbb R}^d \to {\mathbb R}$  with $\|\Phi\|_{\beta,c,p}<\infty$,
\[
\|q*\Phi\|_{\beta,c,p} \le c_{\beta,c,p}\|\Phi\|_{\beta,c,p},
\]
where  $c_{\beta,c,p}$ is decreasing with respect to $\beta$ and $c_{\beta,c,p}\rightarrow 0$ as $\beta \rightarrow \infty$.
\end{prop}

\begin{proof}
Let $p\in [1,1+\alpha/d)$, $c\in [0,\alpha)$ and $\beta>0$.
Let
$\Phi=\Phi(t,x): \Omega \times [0,\infty) \times {\mathbb R}^d \to {\mathbb R}$ be a measurable function
with $\|\Phi\|_{\beta,c,p}<\infty$.
Since
\begin{align*}
(q*\Phi)(t,x)
&=\rho \int_{[0,t)\times {\mathbb R}^d} q_{t-s}(x-y) \Phi(s,y)\, W(\d s, \d y)
+b \int_{[0,t)\times {\mathbb R}^d} q_{t-s}(x-y) \Phi(s,y)\, \d s\, \d y \\
&\quad+\int_{[0,t)\times {\mathbb R}^d\times {\mathbb R}} q_{t-s}(x-y)\Phi(s,y)z\, (\mu-\nu)(\d s, \d y, \d z),
\end{align*}
we have by the triangle inequality,
\begin{align*}
&\|(q*\Phi) (t,x)\|_p \\
&\le \rho \left\| \int_{[0,t)\times {\mathbb R}^d} q_{t-s}(x-y) \Phi(s,y)\, W(\d s, \d y) \right\|_p
+|b| \left\|\int_{[0,t)\times {\mathbb R}^d} q_{t-s}(x-y) \Phi(s,y)\, \d s\, \d y \right\|_p \\
&\quad +\left\| \int_{[0,t)\times {\mathbb R}^d\times {\mathbb R}} q_{t-s}(x-y)\Phi(s,y)z\, (\mu-\nu)(\d s, \d y, \d z) \right\|_p \\
&=:I_1(t,x) + I_2(t,x) + I_3(t,x).
\end{align*}

On the other hand,
\begin{equation}\label{eq:triangle}
\frac{1+|x|}{1+|y|} \le \frac{1+|x-y|+|y|}{1+|y|} \le 1+|x-y|, \quad x,y\in \R^d.
\end{equation}
This along with Lemma \ref{lem:int-bound} yields that for any $r\in [1, 1+\alpha/d)$ (thanks to the condition that $c\in [0,\alpha)$ and $p\ge1$),
\begin{equation}\label{eq:bound-1}
\begin{split}
&\left( \int_{[0,t)\times {\mathbb R}^d} q_{t-s}(x-y)^r \|\Phi(s,y)\|_p^r \, \d s\, \d y \right)^{1/r} \\
&= e^{\beta t} (1+|x|)^{-c} \\
&\quad \times \left( \int_{[0,t)\times {\mathbb R}^d} q_{t-s}(x-y)^r e^{-r\beta(t-s)}\left(\frac{1+|x|}{1+|y|}\right)^{cr}
\left\{ e^{-\beta s}(1+|y|)^c\|\Phi(s,y)\|_p \right\}^r \, \d s\, \d y \right)^{1/r} \\
&\le   e^{\beta t} (1+|x|)^{-c} \|\Phi\|_{\beta,c,p}
\left( \int_{[0,t)\times {\mathbb R}^d} q_{t-s}(x-y)^r e^{-r\beta(t-s)} (1+|x-y|)^{cr} \, \d s\, \d y \right)^{1/r} \\
&= e^{\beta t} (1+|x|)^{-c} \|\Phi\|_{\beta,c,p}
\left( \int_{[0,t)\times {\mathbb R}^d} \{e^{-\beta s} (1+|w|)^c q_s(w)\}^r \, \d s\, \d w \right)^{1/r} \\
&\lesssim  I(\beta,c,r)^{1/r}e^{\beta t} (1+|x|)^{-c} \|\Phi\|_{\beta,c,p}.
\end{split}
\end{equation}

In the following, let us estimate $I_i(t,x) \ (i=1,2,3)$. First, if $p<2$, then $\rho=0$ by assumption, and so $I_1(t,x)=0$.
Suppose now that  $p\in [2,1+\alpha/d)$.
This yields $d=1$ and $1< \alpha<2$ because $2<1+\alpha/d$.
Then, by the maximal inequality of Burkholder-Davis-Gundy (\cite[Theorem 1]{MR14}) and Minkowski's integral inequality as well as \eqref{eq:bound-1} with $r=2$,
there exists $c_p^{(1)}>0$ such that for any $(t,x)\in [0,\infty)\times {\mathbb R}^d$,
\begin{align*}
I_1(t,x)
&\le \rho c_p^{(1)} E\left[\left( \int_{[0,t)\times {\mathbb R}^d}q_{t-s}(x-y)^2 \Phi(s,y)^2 \, \d s\, \d y \right)^{p/2}\right]^{1/p} \\
&\le \rho c_p^{(1)} \left( \int_{[0,t)\times {\mathbb R}^d} q_{t-s}(x-y)^2 \|\Phi(s,y)\|_p^2 \, \d s\, \d y \right)^{1/2} \\
&\lesssim c_p^{(1)}I(\beta,c,2)^{1/2} e^{\beta t} (1+|x|)^{-c} \|\Phi\|_{\beta,c,p},
\end{align*}
where in the second inequality we used the fact $p\ge 2$.
Therefore,
$$
\| I_1 \|_{\beta,c,p}\lesssim c_p^{(1)}I(\beta,c,2)^{1/2} \|\Phi\|_{\beta,c,p}.
$$

Secondly, according to Minkowski's integral inequality and \eqref{eq:bound-1} with $r=1$,
\begin{align*}
I_2(t,x)\asymp \left\|\int_{[0,t)\times {\mathbb R}^d} q_{t-s}(x-y) \Phi(s,y)\, \d s\, \d y \right\|_p
&\le \int_0^t \left( \int_{{\mathbb R}^d} q_{t-s}(x-y) \|\Phi(s,y)\|_p \, \d y \right) \d s \\
&\lesssim I(\beta,c,1)e^{\beta t} (1+|x|)^{-c} \|\Phi\|_{\beta,c,p},
\end{align*}
and so
$$
\| I_2 \|_{\beta,c,p}\lesssim I(\beta,c,1) \|\Phi\|_{\beta,c,p}.
$$

We finally evaluate $I_3(t,x)$.
We first assume that $1\le p\le 2$.
Then, by \cite[Theorem 1]{MR14} and \eqref{eq:bound-1} with $r=p$, we have
\begin{align*}
I_3(t,x)
&\le c_p^{(2)}
E\left[
\int_{[0,t)\times {\mathbb R}^d\times {\mathbb R}} q_{t-s}(x-y)^p |\Phi(s,y)|^p |z|^p \, \d s \, \d y \, \lambda(\d z)
\right]^{1/p} \\
&=c_p^{(2)} \left(\int_{{\mathbb R}}|z|^p \,\lambda(\d z)\right)^{1/p}
\left(
\int_{[0,t)\times {\mathbb R}^d} q_{t-s}(x-y)^p \|\Phi(s,y)\|_p^p  \, \d s \d y
\right)^{1/p} \\
&\lesssim c_p^{(2)}I(\beta,c,p)^{1/p} e^{\beta t} (1+|x|)^{-c} \|\Phi\|_{\beta,c,p},
\end{align*}
which yields
$$
\| I_3 \|_{\beta,c,p}\lesssim c_p^{(2)}I(\beta,c,p)^{1/p}  \|\Phi\|_{\beta,c,p}.
$$
We next assume that $2\le p<1+\alpha/d$. Then, $d=1$ and $\alpha>1$.
According to \cite[Theorem 1]{MR14}, we have
\begin{align*}
I_3(t,x)
&\le c_p^{(3)}E\left[ \left(
\int_{[0,t)\times {\mathbb R}^d\times {\mathbb R}} q_{t-s}(x-y)^2 \Phi(s,y)^2 z^2 \, \d s \, \d y \, \lambda(\d z) \right)^{p/2}
\right]^{1/p}  \\
&\quad + c_p^{(3)} E\left[
\int_{[0,t)\times {\mathbb R}^d\times {\mathbb R}} q_{t-s}(x-y)^p |\Phi(s,y)|^p |z|^p \, \d s \, \d y \, \lambda(\d z)
\right]^{1/p}\\
&=:J_1(t,x)+J_2(t,x).
\end{align*}
Similarly to the calculations for $I_1(t,x)$ and $I_2(t,x)$,
we have
\[
\| J_1 \|_{\beta,c,p}\lesssim c_p^{(4)}I(\beta,c,2)^{1/2} \|\Phi\|_{\beta,c,p}, \quad
\| J_2 \|_{\beta,c,p}\le c_p^{(5)}I(\beta,c,p)^{1/p}  \|\Phi\|_{\beta,c,p}
\]
and so
$$
\| I_3 \|_{\beta,c,p}\le c_p^{(6)}(I(\beta,c,2)^{1/2}+I(\beta,c,p)^{1/p}) \|\Phi\|_{\beta,c,p}.
$$

Putting all the estimates above together,
we arrive at the desired result.
\end{proof}

\begin{proof}[Proof of Theorem {\rm \ref{thm:ex-un}}]
Assume the setting in the statement of the theorem. In the following, let $p\in [1,1+\alpha/d)$, $c\in [0,\alpha)$ and $\beta>0$.

We first prove (1).
Suppose that $X=X(t,x)$ and $Y=Y(t,x)$ are mild solutions to \eqref{eq:mild}.
Since
\[
X(t,x)-Y(t,x)=\left(q*(\sigma(X)-\sigma(Y))\right)(t,x),
\]
it follows from Proposition \ref{prop:contract} and \eqref{eq:lip} that
\begin{equation}\label{eq:fshe-lip}
\|X-Y\|_{\beta,c,p}=\|q*(\sigma(X)-\sigma(Y))\|_{\beta,c,p}
\le c_{\beta,c,p}\|\sigma(X)-\sigma(Y)\|_{\beta,c,p}
\le L_\sigma c_{\beta,c,p}\|X-Y\|_{\beta,c,p},
\end{equation}
where $c_{\beta,c,p}$ is decreasing with respect to $\beta$ such that $c_{\beta,c,p}\rightarrow 0$ as $\beta\rightarrow\infty$.
In particular, one can take a positive constant
$
\beta_0:=\beta_0(d,\alpha,p,c)
$ such that
\begin{equation}\label{e:eee}
c_{\beta_0,c,p}\le (2L_\sigma)^{-1}.
\end{equation}
Then, for all $\beta\ge \beta_0$,
\[
\|X-Y\|_{\beta,c,p}\le \frac{1}{2}\|X-Y\|_{\beta,c,p}.
\]
This yields $\|X-Y\|_{\beta,c,p}=0$ and so $X=Y$.

We next prove (2) by the Picard iteration.
For $n\ge 1$ and $(t,x)\in (0,\infty) \times {\mathbb R}^d$, define
\begin{align*}
\left\{
\begin{aligned}
X^0(t,x)
&=\int_{{\mathbb R}^d} q_t(x-y)u_0(y)\,\d y, \\
X^{n+1}(t,x)
&=X^{0}(t,x)+(q*\sigma(X^n))(t,x).
\end{aligned}
\right.
\end{align*}
Let $\beta_0$ be as in the proof of part (1).
Then, as for \eqref{eq:fshe-lip},
we have for all $\beta\ge \beta_0$,
\begin{equation}\label{eq:fshe-lip-1}
\|X^{n+1}-X^n\|_{\beta,c,p} \le \frac{1}{2}\|X^{n}-X^{n-1}\|_{\beta,c,p}
\le \frac{1}{2^n}\|X^1-X^0\|_{\beta,c,p}
=\frac{1}{2^n}\|q*\sigma(X^0)\|_{\beta,c,p}.
\end{equation}
Let us calculate the right hand side above.

\begin{enumerate}
\item Assume that $\sigma(0)=0$ and
$\sup_{y\in {\mathbb R}^d}\{(1+|y|)^c|u_0(y)|\}<\infty$ for some $c\in [0,\alpha)$.
Then, as in \eqref{eq:fshe-lip} again,
\begin{equation}\label{eq:fshe-lip-2}
\begin{split}
\|q*\sigma(X^0)\|_{\beta,c,p}
&=\|q*(\sigma(X^0)-\sigma(0))\|_{\beta,c,p}\\
&\le \frac{1}{2}\|X^0\|_{\beta,c,p}\\
&=\frac{1}{2}\sup_{t\ge 0, \, x\in {\mathbb R}^d}
\left(e^{-\beta t}(1+|x|)^c \left|\int_{{\mathbb R}^d} q_t(x-y)u_0(y)\, \d y\right|\right) \\
&\le \frac{1}{2}\sup_{t\ge 0, \, x\in {\mathbb R}^d}
\left(e^{-\beta t}(1+|x|)^c  \int_{{\mathbb R}^d} q_t(x-y)|u_0(y)|\, \d y \right).
\end{split}
\end{equation}
By \eqref{eq:triangle} and $(1+|x-y|)^c \le 2^c(1+|x-y|^c)$ for any $c\ge 0$,
we have
\begin{equation*}
\begin{split}
&e^{-\beta t}(1+|x|)^c  \int_{{\mathbb R}^d} q_t(x-y)|u_0(y)|\, \d y
=e^{-\beta t} \int_{{\mathbb R}^d} q_t(x-y) \frac{(1+|x|)^c }{(1+|y|)^c} (1+|y|)^c |u_0(y)|\, \d y \\
&\le 2^c e^{-\beta t} \sup_{y\in {\mathbb R}^d}\{(1+|y|)^c |u_0(y)|\}\int_{{\mathbb R}^d} q_t(x-y) (1+|x-y|^c) \, \d y.
\end{split}
\end{equation*}
Note that
$\int_{{\mathbb R}^d}q_t(x-y)\,\d y=1$ and
\begin{equation*}
\begin{split}
\int_{{\mathbb R}^d}q_t(x-y)|x-y|^c\, \d y
=\int_{{\mathbb R}^d}q_t(z)|z|^c\, \d z
&\asymp \int_{|z|\le t^{1/\alpha}}\frac{|z|^c}{t^{d/\alpha}}\, \d z
+\int_{|z|>t^{1/\alpha}}\frac{t}{|z|^{d+\alpha}}|z|^c\, \d z\\
&=
\begin{dcases}
\omega_d\left(\frac{1}{c+d}+\frac{1}{\alpha-c}\right) t^{c/\alpha}, & c<\alpha, \\
\infty, & c\ge \alpha.
\end{dcases}
\end{split}
\end{equation*}
Hence, if $c\in [0,\alpha)$ and $\sup_{y\in {\mathbb R}^d}\{(1+|y|)^c|u_0(y)|\}<\infty$,
then
\[
e^{-\beta t}(1+|x|)^c  \int_{{\mathbb R}^d} q_t(x-y)|u_0(y)|\, \d y
\le c_1e^{-\beta t}t^{c/\alpha}.
\]
In particular, according to \eqref{eq:fshe-lip-1} and \eqref{eq:fshe-lip-2},
we have
\[
\|X^{n+1}-X^n\|_{\beta,c,p} \le \frac{c_2}{2^{n}},
\]
whence the sequence $\{X^n\}$ is convergent in $\|\cdot\|_{\beta,c,p}$.
Furthermore, Proposition \ref{prop:contract} implies that
the limit $X=\lim_{n\rightarrow\infty}X^n$ in $\|\cdot\|_{\beta,c,p}$ is a mild solution to the equation \eqref{eq:fshe}.

\item Assume that $\sigma(0)\ne 0$ and $c=0$.
Then, by the triangle inequality and the argument similar to (i),
we obtain
\begin{equation}\label{eq:fshe-lip-3}
\begin{split}
\|q*\sigma(X^0)\|_{\beta,0,p}
&=\|q*(\sigma(X^0)-\sigma(0))+q*(\sigma(0)) \|_{\beta,0,p}\\
&\le \|q*(\sigma(X^0)-\sigma(0))\|_{\beta,0,p} + |\sigma(0)| \|q*1\|_{\beta,0,p} \\
&\le c_3.
\end{split}
\end{equation}
Hence, by \eqref{eq:fshe-lip-1} and \eqref{eq:fshe-lip-3},
\[
\|X^{n+1}-X^n\|_{\beta,0,p}\le \frac{c_3}{2^n}.
\]
Then, as in (i), the sequence $\{X^n\}$ is convergent in $\|\cdot\|_{\beta,0,p}$,
and the limit $X=\lim_{n\rightarrow\infty}X^n$ in $\|\cdot\|_{\beta,0,p}$
is a mild solution to the equation \eqref{eq:fshe}.
\end{enumerate}

Therefore, according to the argument above, we obtain the assertion.
\end{proof}

\section{Lyapunov exponents of order $p$}\label{section3}
In this section, we present bounds on
Lyapunov exponents for
the mild solution to \eqref{eq:fshe}.
Our approach is similar to that of \cite[Theorem 2.4(1) and Theorem 3.5]{CK19}.
From this section, we always assume that \emph{the conditions in Theorem {\rm\ref{thm:ex-un}(2)} hold}, and let $X=X(t,x)$ be
the unique mild solution to \eqref{eq:fshe}.

\subsection{Upper bounds}

\begin{prop}\label{prop:growth-upper}
Let $p\in [1,1+\alpha/d)$ and $c\in [0,\alpha)$.
Assume in addition that $\rho=0$ if $p<2$.
Let $\beta_0$ be a positive constant satisfying \eqref{e:eee} as in the proof of Theorem $\ref{thm:ex-un}${\rm (1)}.
Under either of the conditions {\rm (i)} and {\rm (ii)} in Theorem {\rm \ref{thm:ex-un}}{\rm (2)}, it holds that
$\overline{\gamma}(p)\le p\beta_0$.
\end{prop}

\begin{proof}
Under the assumptions of the proposition, we have  $\|X\|_{\beta,0,p}\le \|X\|_{\beta,c,p}<\infty$
for any $\beta\ge \beta_0$. Then
\begin{equation*}
\begin{split}
\overline{\gamma}(p)
&=\limsup_{t\rightarrow\infty}\frac{1}{t}
\sup_{x\in {\mathbb R}^d} \log \left(e^{-p\beta_0 t}E[|X(t,x)|^p]\right)+p\beta_0 \\
&\le \limsup_{t\rightarrow\infty}\frac{1}{t}\log \|X\|^p_{\beta_0, 0, p}+p\beta_0=p\beta_0,
\end{split}
\end{equation*}
which completes the proof.
\end{proof}

\subsection{Lower bounds}

\begin{prop}\label{thm:growth-lower}
Assume that all of the following conditions are
satisfied{\rm :}
\begin{itemize}
\item $b=0$.
\item
$\inf_{t>0, \, y\in \R^d}\int_{\R^d} q_t(y-z)u_0(z)\, \d z >0$.
\item
$L_{\sigma,0}:=\inf_{w\in {\mathbb R}\setminus \{0\}}|\sigma(w)|/|w|>0$.
\end{itemize}
Then the assertions below hold.
\begin{enumerate}
\item If $\alpha>d=1$, then for any $p\in (1, 1+\alpha/d)$, $\underline{\gamma}(p)>0$.
\item If $d\ge\alpha$, then there exists
$p_0\in (1,1+\alpha/d)$
such that for any
$p_0\in (p_0,1+\alpha/d)$,
$\underline{\gamma}(p)>0$.
\item For any $p\in (1,1+\alpha/d)$, there is a constant $L(p)>0$ such that for all $L_{\sigma,0}>L(p)$, $\underline{\gamma}(p)>0$.
\end{enumerate}
\end{prop}

For the proof of Proposition \ref{thm:growth-lower},
we calculate the moment of $q_t(x){\bf 1}_{\{q_t(x)>\varepsilon\}}$.
\begin{lem}\label{lem:h-moment}
For any $\varepsilon>0$ and $p\in [0,1+\alpha/d)$,
\[
\frac{C_1^{1+\alpha/d}c_{d,\alpha}^{(p)}}{\varepsilon^{1+\alpha/d-p}}
\le \int_0^{\infty}\left(\int_{\R^d}q_t(y)^p {\bf 1}_{\{q_t(
y)>\varepsilon\}}\,\d y\right)\,\d t
\le \frac{C_2^{1+\alpha/d}c_{d,\alpha}^{(p)}}{\varepsilon^{1+\alpha/d-p}}
\]
with
\[
c_{d,\alpha}^{(p)}=\frac{\omega_d (d+\alpha)^2}{d(2d+\alpha)(d+\alpha-pd)}.
\]
\end{lem}

\begin{proof}
We give an upper bound of the assertion only
because the lower bound follows in the same way.
By \eqref{eq:g-decay},
\begin{equation}\label{eq:h-moment-1}
\begin{split}
&\int_{\R^d} q_t(y)^p {\bf 1}_{\{q_t(y)>\varepsilon\}}\,\d y\\
&=\int_{|y|\le t^{1/\alpha}} q_t(y)^p {\bf 1}_{\{q_t(y)>\varepsilon\}}\,\d y
+\int_{|y|> t^{1/\alpha}} q_t(y)^p {\bf 1}_{\{q_t(y)>\varepsilon\}}\,\d y\\
&\le C_2^p \int_{|y|\le t^{1/\alpha}} (t^{-d/\alpha})^p {\bf 1}_{\{t<(C_2/\varepsilon)^{\alpha/d}\}}\,\d y
+C_2^p \int_{|y|>t^{1/\alpha}} \left(\frac{t}{|y|^{d+\alpha}}\right)^p {\bf 1}_{\{|y|<(C_2 t/\varepsilon)^{1/(d+\alpha)}\}}\,\d y \\
&=\frac{C_2^p\omega_d}{d}t^{(1-p)d/\alpha} {\bf 1}_{\{t<(C_2/\varepsilon)^{\alpha/d}\}}
+C_2^p\omega_d t^p \int_{t^{1/\alpha}}^{(C_2 t/\varepsilon)^{1/(d+\alpha)}}
r^{d-1-p(d+\alpha)}\,\d r{\bf 1}_{\{t<(C_2/\varepsilon)^{\alpha/d}\}}.
\end{split}
\end{equation}

We first assume that $p\in [0,d/(d+\alpha))$.
If $t<(C_2/\varepsilon)^{\alpha/d}$, then
\begin{equation*}
\begin{split}
t^p\int_{t^{1/\alpha}}^{(C_2 t/\varepsilon)^{1/(d+\alpha)}}
r^{d-1-p(d+\alpha)}\,\d r
&=\frac{t^p}{d-p(d+\alpha)}\left\{\left( \frac{C_2 t}{\varepsilon}\right)^{d/(d+\alpha)-p}-t^{-p+(1-p)d/\alpha}\right\}\\
&=\frac{1}{d-p(d+\alpha)}\left\{\left( \frac{C_2}{\varepsilon}\right)^{d/(d+\alpha)-p}t^{d/(d+\alpha)}-t^{(1-p)d/\alpha}\right\}.
\end{split}
\end{equation*}
Therefore,
\begin{equation*}
\begin{split}
&\int_0^{\infty}\left(t^p\int_{t^{1/\alpha}}^{(C_2 t/\varepsilon)^{1/(d+\alpha)}}
r^{d-1-p(d+\alpha)}\,\d r\right){\bf 1}_{\{t<(C_2/\varepsilon)^{\alpha/d}\}}\,\d t \\
&=\frac{1}{d-p(d+\alpha)}\int_0^{(C_2 /\varepsilon)^{\alpha/d}}
\left\{\left( \frac{C_2}{\varepsilon}\right)^{d/(d+\alpha)-p}t^{d/(d+\alpha)}-t^{(1-p)d/\alpha}\right\}\,\d t\\
&= \left(\frac{C_2}{\varepsilon}\right)^{1+\alpha/d-p}\frac{d}{(d+\alpha-pd)(2d+\alpha)}.
\end{split}
\end{equation*}
Then by \eqref{eq:h-moment-1},
\begin{equation*}
\begin{split}
&\int_0^{\infty}\left(\int_{\R^d} q_t(y)^p {\bf 1}_{\{q_t(y)>\varepsilon\}}\,\d y\right)\,\d t\\
&\le \frac{C_2^p\omega_d}{d}\int_0^{\infty}
t^{(1-p)d/\alpha} {\bf 1}_{\{t<(C_2/\varepsilon)^{\alpha/d}\}}\,\d t\\
&\quad +C_2^p\omega_d \int_0^{\infty} \left(t^p \int_{t^{1/\alpha}}^{(C_2 t/\varepsilon)^{1/(d+\alpha)}}
r^{d-1-p(d+\alpha)}\,\d r{\bf 1}_{\{t<(C_2/\varepsilon)^{\alpha/d}\}}\right)\,\d t\\
&=\frac{C_2^{1+\alpha/d}}{\varepsilon^{(1+\alpha/d)-p}}\omega_d \left(\frac{\alpha}{d(d+\alpha-pd)} + \frac{d}{(d+\alpha-pd)(2d+\alpha)}\right)\
=\frac{C_2^{1+\alpha/d}c_{d,\alpha}^{(p)}}{\varepsilon^{(1+\alpha/d)-p}}.
\end{split}
\end{equation*}

We next assume that $p=d/(d+\alpha)$.
If $t<(C_2/\varepsilon)^{\alpha/d}$, then
\begin{equation*}
\begin{split}
t^p\int_{t^{1/\alpha}}^{(C_2 t/\varepsilon)^{1/(d+\alpha)}}r^{d-1-p(d+\alpha)}\,\d r
=t^p\int_{t^{1/\alpha}}^{(C_2 t/\varepsilon)^{1/(d+\alpha)}}r^{-1}\,\d r
&=\frac{t^p}{d+\alpha}\log\left(\frac{C_2 t}{\varepsilon}\right)-\frac{t^p}{\alpha}\log t.
\end{split}
\end{equation*}
Therefore,
\begin{equation*}
\begin{split}
&\int_0^{\infty}\left(t^p\int_{t^{1/\alpha}}^{(C_2 t/\varepsilon)^{1/(d+\alpha)}}
r^{d-1-p(d+\alpha)}\,\d r\right){\bf 1}_{\{t<(C_2/\varepsilon)^{\alpha/d}\}}\,\d t \\
&=\int_0^{(C_2 /\varepsilon)^{\alpha/d}}
\left(\frac{t^p}{d+\alpha}\log\left(\frac{C_2 t}{\varepsilon}\right)-\frac{t^p}{\alpha}\log t\right)\,\d t
= \left(\frac{C_2}{\varepsilon}\right)^{(p+1)\alpha/d}\frac{d}{\alpha(d+\alpha)(p+1)^2}.
\end{split}
\end{equation*}
Then by \eqref{eq:h-moment-1},
\begin{equation*}
\begin{split}
&\int_0^{\infty}\left(\int_{\R^d} q_t(y)^p {\bf 1}_{\{q_t(y)>\varepsilon\}}\,\d y\right)\,\d t\\
&\le \frac{C_2^p\omega_d}{d}\int_0^{\infty}
t^{(1-p)d/\alpha} {\bf 1}_{\{t<(C_2/\varepsilon)^{\alpha/d}\}}\,\d t\\
&\quad+C_2^p\omega_d \int_0^{\infty} \left(t^p \int_{t^{1/\alpha}}^{(C_2 t/\varepsilon)^{1/(d+\alpha)}}
r^{d-1-p(d+\alpha)}\,\d r{\bf 1}_{\{t<(C_2/\varepsilon)^{\alpha/d}\}}\right)\,\d t\\
&=C_2^p \omega_d \left(\frac{\alpha}{d(d+\alpha-pd)} \left(\frac{C_2}{\varepsilon}\right)^{1+\alpha/d-p}
+ \frac{d}{\alpha(d+\alpha)(p+1)^2} \left(\frac{C_2}{\varepsilon}\right)^{(p+1)\alpha/d}\right)\\
&=C_2^p \left(\frac{C_2}{\varepsilon}\right)^{1+\alpha/d-p} \omega_d \left(\frac{\alpha}{d(d+\alpha-pd)}
+ \frac{d}{(2d+\alpha)(d+\alpha-pd)} \right)
=\frac{C_2^{1+\alpha/d}c_{d,\alpha}^{(p)}}{\varepsilon^{(1+\alpha/d)-p}}.
\end{split}
\end{equation*}
At the second equality above, we used the relation $1+\alpha/d-p=(p+1)\alpha/d$.

For $p\in (d/(d+\alpha), 1+\alpha/d)$, we have the desired inequality in the same manner as for $p\in [0,d/(d+\alpha))$.
Hence the proof is complete.
\end{proof}

\begin{proof}[Proof of Proposition {\rm \ref{thm:growth-lower}}]
By the 
H\"{o}lder inequality for $E[|X(t,x)|^p]$, we may and do assume that $p\in (1,2)$.
The proof is split into three steps.

(1)
As $b=0$ by assumption, we see by \eqref{eq:mild} that,
for all $t>0$ and $x\in \R^d$,
\[
X(t,x)=\int_{\R^d} q_t(x-y)u_0(y)\, \d y+\Phi_1(X)(t,x)+\Phi_3(X)(t,x).
\]
Since $\Phi_1(X)(t,x)$ is a continuous local martingale
and $\Phi_3(X)(t,x)$ is a purely discontinuous local martingale,
the Burkholder-Davis-Gundy inequality implies that for any $p>0$,
\begin{equation*}
\begin{split}
&E[|\Phi_1(X)(t,x)+\Phi_3(X)(t,x)|^p]\\
&\asymp
E\Biggl[\Biggl( \int_{[0,t)\times \R^d} q_{t-s}(x-y)^2\sigma(X(s,y))^2\,\d s\d y \\
&\qquad \qquad+\int_{[0,t) \times \R^d \times \R} q_{t-s}(x-y)^2 \sigma(X(s,y))^2 z^2\,\mu(\d s, \d y, \d z)\Biggr)^{p/2}\Biggr]\\
&\ge E\left[\left( \int_{[0,t) \times \R^d \times \R} q_{t-s}(x-y)^2 \sigma(X(s,y))^2 z^2\,\mu(\d s, \d y, \d z)\right)^{p/2}\right].
\end{split}
\end{equation*}
Furthermore, since $\Phi_1(X)(t,x)+\Phi_3(X)(t,x)$ has zero mean and
\[
\inf_{t>0, \, y\in \R^d}\int_{\R^d} q_t(y-z)u_0(z)\, \d z >0,
\]
we have by Lemma \ref{lem:mean-bound}(ii),
\begin{equation}\label{eq:moment-lower}
\begin{split}
&E[|X(t,x)|^p]\\
&\ge
\kappa^{(p)} \left(
\left| \inf_{t>0, \, y\in \R^d}\int_{\R^d} q_t(y-z)u_0(z)\, \d z \right|^p
+E[|\Phi_1(X)(t,x)+\Phi_3(X)(t,x)|^p]\right) \\
&\ge
c_1^{(p)}\left(1+E\left[\left( \int_{[0,t) \times \R^d \times \R} q_{t-s}(x-y)^2 \sigma(X(s,y))^2 z^2\,\mu(\d s, \d y, \d z)\right)^{p/2}\right]\right),
\end{split}
\end{equation}
where $\kappa^{(p)}$ and $c_1^{(p)}$ are positive constants independently of $(t,x)\in (0,\infty) \times \R^d$.

Below, for $\varepsilon>0$ and $\delta>0$ with $\lambda([-\delta,\delta]^c)>0$,
define the measure $\mu_{\varepsilon,\delta}^{(t,x)}$ by
\[
\mu_{\varepsilon,\delta}^{(t,x)}(\d s, \d y, \d z)
={\bf 1}_{\{0\le s<t, \, q_{t-s}(x-y)>\varepsilon\}}{\bf 1}_{\{|z|>\delta\}}\,\mu(\d s, \d y, \d z).
\]
The corresponding intensity measure is given by
\begin{align*}
\nu_{\varepsilon,\delta}^{(t,x)}(\d s, \d y, \d z)
&={\bf 1}_{\{0\le s<t, \, q_{t-s}(x-y)>\varepsilon\}}{\bf 1}_{\{|z|>\delta\}}\,\nu(\d s, \d y, \d z) \\
&={\bf 1}_{\{0\le s<t, \, q_{t-s}(x-y)>\varepsilon\}}{\bf 1}_{\{|z|>\delta\}}\,\d s \d y \lambda(\d z).
\end{align*} In particular,
\begin{equation}\label{eq:nu-1}
\begin{split}
\nu_{\varepsilon,\delta}^{(t,x)}([0,\infty)\times \R^d\times \R)
&=\lambda([-\delta, \delta]^c)
\int_0^t \left(\int_{\R^d} {\bf 1}_{\{q_{t-s}(x-y)>\varepsilon\}}\,\d y \right) \,\d s \\
&=\lambda([-\delta, \delta]^c)
\int_0^t \left(\int_{\R^d} {\bf 1}_{\{q_s(y)>\varepsilon\}}\,\d y \right) \,\d s \\
&\le \lambda([-\delta, \delta]^c)
\int_0^{\infty} \left(\int_{\R^d} {\bf 1}_{\{q_s(y)>\varepsilon\}}\,\d y \right) \,\d s.
\end{split}
\end{equation}
Recall that $p\in (1,2)$ by assumption.
Hence,
according to \eqref{eq:moment-lower} and
the Burkholder-Davis-Gundy inequality, Lemma \ref{lem:PG-ineq} and \eqref{eq:nu-1},
\begin{equation}\label{eq:p-norm}
\begin{split}
&E[|X(t,x)|^p]\\
&\ge
c_1^{(p)}\left(1+E\left[\left( \int_{[0,t) \times \R^d \times \R}
q_{t-s}(x-y)^2 \sigma(X(s,y))^2 z^2\,\mu_{\varepsilon,\delta}^{(t,x)}(\d s, \d y, \d z)\right)^{p/2}\right]\right) \\
&\ge
c_2^{(p)}\left(1+E\left[\left| \int_{[0,t) \times \R^d \times \R}
q_{t-s}(x-y) \sigma(X(s,y)) z\,(\mu_{\varepsilon,\delta}^{(t,x)}-\nu_{\varepsilon,\delta}^{(t,x)})(\d s, \d y, \d z)\right|^p\right]\right) \\
&\ge c_3^{(p)} \left(1+
\frac{\int_{[0,\infty)\times \R^d \times \R} q_{t-s}(x-y)^p |z|^p E[|\sigma(X(s,y))|^p] \,
\nu_{\varepsilon, \delta}^{(t,x)}(\d s, \d y, \d z)}
{\left\{1 \vee \nu_{\varepsilon,\delta}^{(t,x)}([0,\infty) \times \R^d \times \R)\right\}^{1-p/2}}\right) \\
&\ge c_3^{(p)} \left(1+
\frac{\int_{|z|>\delta} |z|^p \,\lambda(\d z)
\int_{[0,t)\times \R^d} q_{t-s}(x-y)^p {\bf 1}_{\{q_{t-s}(x-y)>\varepsilon\}}  E[|\sigma(X(s,y))|^p] \, \d s \d y}
{\left\{1 \vee \lambda([-\delta,\delta]^c)
\int_{[0,\infty) \times \R^d} {\bf 1}_{\{q_s(y)>\varepsilon\}} \, \d s \d y\right\}^{1-p/2}}\right).
\end{split}
\end{equation}
We here note that for any $r_0\in (1,2]$, $\inf_{p\in [r_0,2]}c_3^{(p)}>0$.

(2) By Lemma \ref{lem:h-moment} with $p=0$, we have
\begin{equation}\label{eq:tail-1}
C_1^{1+\alpha/d}  \frac{\omega_d(d+\alpha)}{d(2d+\alpha)}\frac{1}{\varepsilon^{1+\alpha/d}}
\le \int_0^{\infty} \left(\int_{\R^d} {\bf 1}_{\{q_s(y)>\varepsilon\}}\,\d y \right) \,\d s
\le C_2^{1+\alpha/d}\frac{\omega_d(d+\alpha)}{d(2d+\alpha)}\frac{1}{\varepsilon^{1+\alpha/d}}.
\end{equation}

On the other hand, since, by assumption,
\[
L_{\sigma,0}:=\inf_{z\in \R\setminus \{0\}}\frac{|\sigma(z)|}{|z|}>0,
\] it holds that
\begin{align*}
&\int_{[0,t)\times \R^d} q_{t-s}(x-y)^p {\bf 1}_{\{q_{t-s}(x-y)>\varepsilon\}}
E[|\sigma(X(s,y))|^p] \, \d s \d y \\
&\ge L_{\sigma,0}^p \int_{[0,t)\times \R^d} q_{t-s}(x-y)^p {\bf 1}_{\{q_{t-s}(x-y)>\varepsilon\}} E[|X(s,y)|^p] \, \d s \d y \\
&\ge L_{\sigma,0}^p \int_{[0,t)\times \R^d} q_{t-s}(x-y)^p {\bf 1}_{\{q_{t-s}(x-y)>\varepsilon\}} \inf_{y\in \R^d}E[|X(s,y)|^p] \, \d s \d y \\
&= L_{\sigma,0}^p \int_{[0,t)\times \R^d} q_{t-s}(y)^p {\bf 1}_{\{q_{t-s}(y)>\varepsilon\}} \inf_{y\in \R^d}E[|X(s,y)|^p] \, \d s \d y,
\end{align*}
Then, by \eqref{eq:p-norm},
\begin{align*}
&E[|X(t,x)|^p] \\
&\ge
c_3^{(p)} \left(1+ L_{\sigma,0}^p
\frac{\int_{|z|>\delta} |z|^p \,\lambda(\d z)
 \int_{[0,t)\times \R^d} q_{t-s}(y)^p {\bf 1}_{\{q_{t-s}(y)>\varepsilon\}}\inf_{y\in \R^d}E[|X(s,y)|^p] \, \d s \d y }
{\left\{1 \vee \lambda([-\delta,\delta]^c)
\int_{[0,\infty) \times \R^d} {\bf 1}_{\{q_s(y)>\varepsilon\}} \, \d s \d y\right\}^{1-p/2}}\right).
\end{align*}
Hence, if we define
\[
I_p(t)=\inf_{y\in \R^d}E[|X(t,y)|^p]
\]
and
\begin{equation}\label{eq:def-w}
w_p^{(\varepsilon)}(t)
=\frac{\int_{|z|>\delta} |z|^p \,\lambda(\d z) }
{\left\{1 \vee \lambda([-\delta,\delta]^c)
\int_{[0,\infty) \times \R^d} {\bf 1}_{\{q_s(y)>\varepsilon\}} \, \d s \d y\right\}^{1-p/2}}
\times  \int_{\R^d} q_t(y)^p {\bf 1}_{\{q_t(y)>\varepsilon\}} \, \d y,
\end{equation}
then
\begin{equation}\label{eq:super-sol}
I_p(t)\ge c_3^{(p)} +c_4^{(p)}\int_0^t w_p^{(\varepsilon)}(t-s) I_p(s) \, \d s,
\end{equation}
where $c_4^{(p)}=c_3^{(p)}L_{\sigma,0}^p.$

Furthermore, for all $p\in (1, 1+\alpha/d)$, we have by Lemma \ref{lem:h-moment},
\begin{equation}\label{eq:tail-2}
\frac{C_1^{1+\alpha/d} c_{d,\alpha}^{(p)}}{\varepsilon^{1+\alpha/d-p}}
\le \int_0^{\infty}\left(\int_{\R^d} q_t(y)^p {\bf 1}_{\{q_t(y)>\varepsilon\}}(y) \, \d y \right) \, \d t
\le \frac{C_2^{1+\alpha/d} c_{d,\alpha}^{(p)}}{\varepsilon^{1+\alpha/d-p}}.
\end{equation}
Then, by \eqref{eq:tail-1} and \eqref{eq:tail-2},
\begin{equation}\label{eq:int-w}
\int_0^{\infty} w_p^{(\varepsilon)}(t)\, \d t
\asymp \frac{c_{d,\alpha}^{(p)}\varepsilon^{-(1+\alpha/d-p)}}{(\varepsilon^{-(1+\alpha/d)})^{1-p/2}}
=\frac{c_{d,\alpha}^{(p)}}{\varepsilon^{p(\alpha/d-1)/2}}.
\end{equation}

(3) In this part, we will prove the desired assertion cases by cases.

\begin{itemize}

\item[{\rm(i)}] Consider the case that $\alpha>d=1$.
Then, by \eqref{eq:int-w}, there exists $\varepsilon_0>0$ such that for all $p\in (1,1+\alpha/d)$,
\[
c_4^{(p)}\int_0^{\infty}w_p^{(\varepsilon_0)}(t) \,\d t>1.
\]
In particular, there is some $\beta_1>0$ so that
\[
c_4^{(p)} \int_0^{\infty} e^{-\beta_1 t}w_p^{(\varepsilon_0)}(t)\, \d t=1.
\]

On the one hand, let $f(t)$ be a measurable function on $[0,\infty)$ such that
$f$ is bounded on each finite interval in $[0,\infty)$ and satisfies the equation
\[
f(t)=c_3^{(p)}+c_4^{(p)} \int_0^t w_p^{(\varepsilon_0)}(t-s) f(s)\, \d s.
\]
Then, by \cite[p.~162, Theorem 7.1]{A03},
\begin{equation}\label{eq:f-limit}
\lim_{t\rightarrow\infty} e^{-\beta_1 t}f(t)
=\frac{c_3^{(p)}}{\beta_1 c_4^{(p)}\int_0^{\infty}t e^{-\beta_1 t}w_p^{(\varepsilon_0)}(t)\,\d t}
\end{equation}
and so $\sup_{t\ge 0} (e^{-\beta t}|f(t)|)<\infty$ for any $\beta\ge \beta_1$.

On the other hand, recall that for any $\beta\ge \beta_0$, $\|X\|_{\beta,c,p}<\infty$, where $\beta_0$ is given in \eqref{e:eee}.
Since
\[
E [|X(t,x)|^p]^{1/p}=e^{\beta t}(1+|x|)^{-c}\{e^{-\beta t}(1+|x|)^c E[|X(t,x)|^p]^{1/p}\}
\le e^{\beta t} \|X\|_{\beta,c,p},
\]
we obtain  $I_p(t)\le e^{p\beta t}\|X\|_{\beta,c,p}^p$,
and so for any $\beta\ge p\beta_0$,
\[
\sup_{t\ge 0}(e^{-\beta t} I_p(t)) \le
\sup_{t\ge 0}(e^{-\beta t}e^{p\beta_0t})  \|X\|_{\beta_0,c,p}^p
\le \|X\|_{\beta_0,c,p}^p.
\]
Hence, by \cite[Theorem 7.11]{K14}, we have
\begin{equation}\label{e:ppp}
I_p(t)\ge f(t), \quad t\ge 0.
\end{equation}
Combining this with \eqref{eq:f-limit},  we get
\begin{equation*}
\begin{split}
\underline{\gamma}(p)
&=\liminf_{t\rightarrow\infty}\frac{1}{t}\inf_{y\in \R^d}\log E[|X(t,y)|^p]
=\liminf_{t\rightarrow\infty}\frac{1}{t}\log I_p(t) \\
&\ge \liminf_{t\rightarrow\infty}\frac{1}{t}\log f(t) =\beta_1>0.
\end{split}
\end{equation*}
We thus complete the proof for the assertion (i) in the proposition.

\item[{\rm(ii)}] In this part, we consider the case that $d\ge \alpha$.
By \eqref{eq:tail-2}, we have as $p\to 1+\alpha/d$,
$$\int_0^\infty \left(\int_{\R^d}q_t(y)^p {\bf 1}_{\{q_t(y)>1\}} \, \d y\right)\,\d t\to \infty.$$
Combining this with
\eqref{eq:tail-1} and \eqref{eq:def-w}, we obtain
$$\lim_{p\to 1+\alpha/d} \int_0^\infty w^{(1)}_p(t)\,\d t=\infty.$$
Hence, there exists $p_0\in (1,1+\alpha/d)$ such that
\[
c_4^{(p_0)}\int_0^{\infty}w_{p_0}^{(1)}(t) \,\d t>1.
\]
Then by following the arguments in (i),
we have $\underline{\gamma}(p_0)\ge \beta_2$ for some $\beta_2>0$.
Moreover, for any $p\in (p_0,1+\alpha/d)$, the Schwarz inequality yields
\[
E[|X(t,x)|^{p_0}]\le E[|X(t,x)|^p]^{p_0/p}, \quad (t,x)\in [0,\infty)\times \R^d
\]
and so
\[
\underline{\gamma}(p)
\ge \frac{p}{p_0}\underline{\gamma}(p_0)
\ge \underline{\gamma}(p_0)\ge \beta_2.
\]
This proves the desired assertion (ii) in the proposition.

\item[{\rm(iii)}]
We start from the estimate \eqref{eq:int-w}.
Recall that $c_4^{(p)}=c_3^{(p)}L_{\sigma,0}^p$.
Then for any $p\in (1,1+\alpha/d)$, there exists $
L(p)>0$ such that for all $L_{\sigma,0}>L(p)$,

With this at hand, we can follow the arguments in (i) to obtain the desired assertion (iii) in the proposition.
\end{itemize}
Hence, the proof is complete.
\end{proof}

\begin{rem}\label{rem:decay}\rm
Suppose that the condition in Theorem \ref{thm:ex-un}(2)(i) is fulfilled with $c\in (0,\alpha)$.
Since there exists $c_1>0$ such that $|u_0(x)|\le c_1(1+|x|)^{-c}$ for any $x\in \R^d$,
we see that for every $t>0$,
the function $Q_tu_0(x):=\int_{\R^d}q_t(x-y) u_0(y)\,\d y$ satisfies
$\lim_{|x|\rightarrow\infty}Q_tu_0(x)=0$.
In particular, if we assume in addition that $u_0$ is nonnegative,
then $\inf_{t>0, \, x\in \R^d}Q_t u_0(x)=0$ and so the condition in Proposition \ref{thm:growth-lower} fails.
Moreover, Theorem \ref{thm:ex-un}(2)(i) yields that for some $c_2>0$,
\[
E[|X(t,x)|^p]\le c_2e^{p\beta t}(1+|x|)^{-pc}, \quad (t,x)\in (0,\infty)\times \R^d.
\]
This implies $\inf_{x\in \R^d}E[|X(t,x)|^p]=0$ and so $\underline{\gamma}(p)=-\infty$ for any $p\in [1,1+\alpha/d)$.
This also partly indicates  that the second condition in Proposition \ref{thm:growth-lower} is indispensable.
\end{rem}

\begin{proof}[Proof of Theorem {\rm \ref{thm1}(i)}]
If the L\'evy measure $\nu$ is symmetric, then
\[
b=\int_{\R}z\,\lambda(\d z)=0.
\]
Since $\int_{\R^d}q_t(x-y)\,\d y=1$, we see that if $\inf_{x\in \R}u_0(x)>0$, then
\[
\inf_{t>0, \, x\in \R}\int_{\R^d}q_t(x-y)u_0(y)\,\d y
\ge \inf_{y\in \R}u_0(y)>0.
\]
Therefore,
Propositions \ref{prop:growth-upper} and \ref{thm:growth-lower} yield the desired assertion.
\end{proof}

\section{Growth indices of exponential type}\label{section4}

Let $p\in [1,1+\alpha/d)$. As mentioned in Remark \ref{rem:decay},
if the initial function $u_0$ is nonnegative and vanishes at infinity with some polynomial rate,
then for each fixed $t>0$, the $p$th moment of the mild solution to \eqref{eq:fshe} also vanishes at infinity.
Here we are concerned with the supremum of the moment for the unique mild solution in the time-dependent spatial region
(which is away from the origin).
Recall that we assume that \emph{the conditions in Theorem {\rm\ref{thm:ex-un}(2)} hold}, and let $X=X(t,x)$ be
the unique mild solution to \eqref{eq:fshe}.

\subsection{Upper bounds}\label{section4.1}

\begin{prop}\label{prop:exp-upper}
Let $p\in [1,1+\alpha/d)$ and $c\in (0,\alpha)$.
Assume that $\rho=0$ if $p<2$.
Let $\beta_0$ be a positive constant satisfying \eqref{e:eee} as in the proof of Theorem {\rm \ref{thm:ex-un}(1)}.
If $\sigma(0)=0$ and $\sup_{y\in {\mathbb R}^d}\{(1+|y|)^c|u_0(y)|\}<\infty$,
then $\overline{\lambda}(p)\le \beta_0/c$.
\end{prop}

\begin{proof}
Under the setting of the proposition, since $\|X\|_{\beta,c,p}<\infty$ for any $\beta\ge \beta_0$,
there exists $c_1>0$ such that
\[
E[|X(t,x)|^p]\le c_1 e^{p\beta_0 t}(1+|x|)^{-pc}, \quad (t,x)\in [0,\infty) \times {\mathbb R}^d.
\]
Then for any $\eta>\beta_0/c$,
\[
\limsup_{t\rightarrow\infty}\frac{1}{t}\sup_{|x|\ge e^{\eta t}} \log E[|X(t,x)|^p]
\le -p(\eta c-\beta_0)<0,
\]
which implies $\overline{\lambda}(p)\le \beta_0/c$.
\end{proof}

\begin{rem}\rm
Assume that $\alpha>d=1$, $b=0$, $\inf_{y\in \R^d} u_0(y)>0$ and $L_{\sigma,0}>0$.
Then by Proposition \ref{thm:growth-lower}(i), we have $\underline{\gamma}(2)>0$
and so $\overline{\lambda}(2)=\underline{\lambda}(2)=\infty$.
Hence we can not drop the condition $c>0$ in Proposition \ref{prop:exp-upper}.
\end{rem}

\subsection{Lower bounds: $p\ge2$}
\begin{prop}\label{prop:p=2}
Assume that all of the following conditions are satisfied{\rm :}
\begin{itemize}
\item $b=0$.
\item
$u_0(x)\ge 0$ for all $x\in \mathbb R^d$, and it is strictly positive on a set of positive Lebesgue measure.
\item
$L_{\sigma,0}:=\inf_{w\in {\mathbb R}\setminus \{0\}}|\sigma(w)|/|w|>0$.
\end{itemize}
If $\alpha>d=1$, then $\underline{\lambda}(p)>0$ for all $p\in [2,1+\alpha/d)$.
\end{prop}

To prove Proposition \ref{prop:p=2}, we follow the argument of \cite[Theorem 3.6]{CD15-1}.
Let us observe a lower bound of the $p$th moment of $X(t,x)$.
Let $\{Q_t\}_{t>0}$ be the convolution semigroup associated with
the heat kernel
$q_t(x)$ and
\[
Qu(t,x):=Q_tu(x)=\int_{\R^d} q_t(x-y)u(y)\, \d y
\]
for any nonnegative Borel measurable function $u$ on $\R^d$.
Since $b=0$ by assumption, we obtain by \eqref{eq:mild}, for all $t>0$ and $x\in \R^d$,
\[
X(t,x)=Q_tu_0(x)+\Phi_1(X)(t,x)+\Phi_3(X)(t,x).
\]
Since $\Phi_1(X)(t,x)$ is a continuous local martingale
and $\Phi_3(X)(t,x)$ is a purely discontinuous local martingale,
we know that $E[\Phi_1(X)(t,x)]=0$ and $E[\Phi_3(X)(t,x)]=0$.
Hence by Lemma \ref{lem:mean-bound}, we have for any $p\in (1,1+\alpha/d)$,
\begin{equation}\label{eq:p-moment}
\begin{split}
E[|X(t,x)|^p]
&=E\left[\left|Qu_0(t,x)+\Phi_1(X)(t,x)+\Phi_3(X)(t,x)\right|^p\right]\\
&\ge
\frac{1}{4}\left(|Qu_0(t,x)|^p+E[|\Phi_1(X)(t,x)+\Phi_3(X)(t,x)|^p]\right).
\end{split}
\end{equation}
Moreover, the Burkholder-Davis-Gundy inequality implies that for any $p\in (1,1+\alpha/d)$,
\begin{equation}\label{eq:p-moment-2}
\begin{split}
&E[|\Phi_1(X)(t,x)+\Phi_3(X)(t,x)|^p]\\
&\asymp
E\Biggl[\Biggl( \int_{[0,t)\times \R^d} q_{t-s}(x-y)^2\sigma(X(s,y))^2\,\d s\d y \\
&\qquad \qquad+\int_{[0,t) \times \R^d \times \R} q_{t-s}(x-y)^2 \sigma(X(s,y))^2 z^2\,\mu(\d s, \d y, \d z)\Biggr)^{p/2}\Biggr]\\
&\ge E\left[\left( \int_{[0,t) \times \R^d \times \R} q_{t-s}(x-y)^2 \sigma(X(s,y))^2 z^2\,\mu(\d s, \d y, \d z)\right)^{p/2}\right].
\end{split}
\end{equation}

In the following, we will assume that $\alpha>d=1$ and $p\in [2,1+\alpha/d)$.
Then, by the maximal inequality for purely discontinuous martingales (\cite[Theorem 1]{MR14}),
we obtain
\begin{equation}\label{eq:BDG-p=2}
\begin{split}
&E\left[\left( \int_{[0,t) \times \R^d \times \R} q_{t-s}(x-y)^2 \sigma(X(s,y))^2 z^2\,\mu(\d s, \d y, \d z)\right)^{p/2}\right]\\
&\gtrsim E\left[\left( \int_{[0,t) \times \R^d \times \R} q_{t-s}(x-y)^p |\sigma(X(s,y))|^p |z|^p\,\nu(\d s, \d y, \d z)\right)\right]\\
&=\sigma_{\lambda}^{(p)}\int_0^t\left(\int_{\R^d}q_{t-s}(x-y)^p E\left[|\sigma(X(s,y))|^p\right]\,\d y\right)\d s
\end{split}
\end{equation}
with $\sigma_{\lambda}^{(p)}=\int_{\R}|z|^p\,\lambda(\d z)$.
Hence, by \eqref{eq:p-moment} and the two inequalities above, we get for some $c_1>0$,
\begin{equation}\label{eq:iteration-0}
E[|X(t,x)|^p]
\ge \frac{1}{4} (Qu_0(t,x))^p
+c_1 \sigma_{\lambda}^{(p)}\int_0^t\left(\int_{\R^d}q_{t-s}(x-y)^p E\left[|\sigma(X(s,y))|^p\right]\,\d y\right)\d s.
\end{equation}

For any nonnegative measurable functions $h_1$ and $h_2$ defined on $(0,\infty)\times \R^d$,
we define their space convolution and time-space convolution, respectively, by
\[
(h_1(u,\cdot)*h_2(v,\cdot))(x)=\int_{\R^d}h_1(u,x-y)h_2(v,y)\,\d y, \quad u,v>0, \ x\in \R^d
\]
and
\[
(h_1\star h_2)(t,x)=\int_0^t\left(\int_{\R^d}h_1(t-s,x-y)h_2(s,y)\,\d y\right)\,\d s, \quad t>0, \, x\in \R^d.
\]
We can inductively define
the $n$th space convolution and the $n$th time-space convolution of nonnegative measurable functions $h_1,\dots, h_n$ on $(0,\infty)\times \R^d$.

Let $c_*^{(p)}=c_1 \sigma_{\lambda}^{(p)} L_{\sigma,0}^p$. We then have
\begin{lem}\label{lem:p=2}
Under the same setting as in Proposition {\rm \ref{prop:p=2}}, for all $t>0$ and $x\in \R^d$,
\begin{equation}\label{eq:iteration-2}
E[|X(t,x)|^p]
\ge \sum_{k=1}^{\infty} \left(\frac{c_*^{(p)}}{4}\right)^k  (\underbrace{q^p\star \dots \star q^p}_k\star (Q u_0)^p)(t,x).
\end{equation}
\end{lem}
\begin{proof}
Since $L_{\sigma,0}>0$, we have by \eqref{eq:iteration-0},
\begin{equation}\label{eq:iteration-1}
\begin{split}
E[|X(t,x)|^p]
&\ge \frac{1}{4} (Qu_0(t,x))^p
+c_1 \sigma_{\lambda}^{(p)} L_{\sigma,0}^p \int_0^t\left(\int_{\R^d} q_{t-s}(x-y)^p E\left[|X(s,y)|^p\right]\,\d y\right)\d s\\
&= \frac{1}{4} (Qu_0(t,x))^p
+c_*^{(p)} \int_0^t\left(\int_{\R^d}q_{t-s}(x-y)^p E\left[|X(s,y)|^p\right]\,\d y\right)\d s,
\end{split}
\end{equation}
which implies that
\begin{equation*}
\begin{split}
&\int_0^t\left(\int_{\R^d}q_{t-s}(x-y)^p E\left[|X(s,y)|^p\right]\,\d y\right)\d s\\
&\ge \int_0^t\Biggl[\int_{\R^d}q_{t-s}(x-y)^p\\
&\qquad \qquad
\times \left\{\frac{1}{4} (Qu_0(s,y))^p
+c_*^{(p)}\int_0^s\left(\int_{\R^d}q_{s-u}(y-z)^p E\left[|X(u,z)|^p \right]\,\d z\right)\d u\right\}\,\d y\Biggr]\,\d s\\
&=\frac{1}{4}\int_0^t \left(\int_{\R^d}q_{t-s}(x-y)^p (Qu_0(s,y))^p\,\d y\right)\,\d s \\
&\quad +c_*^{(p)} \int_0^t\int_{\R^d} q_{t-s}(x-y)^p  \left(\int_0^s \int_{\R^d}q_{s-u}(y-z)^p E[|X(u,z)|^p]\,\d z\,\d u\right)\,\d y\,\d s.
\end{split}
\end{equation*}
Since
\[
\int_0^t \left(\int_{\R^d}q_{t-s}(x-y)^p (Qu_0(s,y))^p\,\d y\right)\,\d s
=\{q^p\star (Qu_0)^p\}(t,x)
\]
and the Fubini theorem yields
\begin{equation*}
\begin{split}
&\int_0^t\int_{\R^d}q_{t-s}(x-y)^p \left(\int_0^s \int_{\R^d}q_{s-u}(y-z)^p E[|X(u,z)|^p] \,\d z\,\d u\right)\,\d y\,\d s \\
&=\int_0^t \int_{\R^d} \left(\int_0^{t-u}\int_{\R^d} q_{t-u-v}(x-y)^p q_{v}(y-z)^p\,\d y\,\d v\right)E[|X(u,z)|^p]\,\d z\,\d u\\
&=\int_0^t \int_{\R^d} (q^p\star q^p)(t-u,x-z)E[|X(u,z)|^p]\,\d z\,\d u,
\end{split}
\end{equation*}
we get
\begin{equation*}
\begin{split}
&\int_0^t\left(\int_{\R^d}q_{t-s}(x-y)^p E\left[|X(s,y)|^p\right]\,\d y\right)\d s\\
&\ge \frac{1}{4}\{q^p \star (Qu_0)^p\}(t,x)+c_*^{(p)}\int_0^t \int_{\R^d} (q^p\star q^p)(t-u,x-z)E[|X(u,z)|^p]\,\d z\,\d u.
\end{split}
\end{equation*}
Hence by \eqref{eq:iteration-1},
\begin{equation*}
\begin{split}
E[|X(t,x)|^p]
&\ge \frac{1}{4} (Qu_0(t,x))^p
+\frac{c_*^{(p)}}{4}\{q^p\star (Qu_0)^p\}(t,x)\\
&\quad +(c_*^{(p)})^2\int_0^t \int_{\R^d} (q^p\star q^p)(t-u,x-z)E[|X(u,z)|^p]\,\d z\,\d u.
\end{split}
\end{equation*}

Repeating the same iterative operation as above, we obtain
\begin{equation*}
\begin{split}
E[|X(t,x)|^p]
&\ge \frac{1}{4}(Qu_0(t,x))^p
+\sum_{k=1}^n \left(\frac{c_*^{(p)}}{4}\right)^k  (\underbrace{q^p\star \dots \star q^p}_k\star (Q u_0)^p)(t,x)\\
&\quad +(c_*^{(p)})^{n+1}\int_0^t \int_{\R^d} (\underbrace{q^p \star \cdots \star q^p}_{n+1})(t-u,x-z)E[|X(u,z)|^p]\,\d z\,\d u\\
&\ge \sum_{k=1}^n \left(\frac{c_*^{(p)}}{4}\right)^k  (\underbrace{q^p \star \dots \star q^p}_k\star (Q u_0)^p)(t,x).
\end{split}
\end{equation*}
Letting $n\rightarrow\infty$, we arrive at \eqref{eq:iteration-2}.
\end{proof}

Let $\gamma_{d,_\alpha}^{(p)}$ and $\Lambda_{d,\alpha}^{(p)}$ be as in Lemma \ref{lem:g-conv}.
Using Lemma \ref{lem:p=2}, we further calculate the lower bound of $E[|X(t,x)|^p]$.
Let $c_{**}^{(p)}=c_*^{(p)}C_{1,g}^p/4$ and
\begin{equation}\label{eq:c_e}
C_{\varepsilon}=\frac{C_{1,g}\varepsilon^{d/\alpha}}{2^{d+\alpha}\kappa_{d,\alpha}}\int_{\R^d}g(\varepsilon,u)u_0(y)\,\d y,
\end{equation}
where $\kappa_{d,\alpha}$ is defined in \eqref{eq:g-def}.

\begin{lem}
Under the same setting as in Proposition {\rm \ref{prop:p=2}},
for all $t>0$, $x\in \R^d$ and $\varepsilon\in (0,t]$,
\begin{equation}\label{eq:iteration-4}
\begin{split}
E[|X(t,x)|^p]
&\ge c_{**}^{(p)}C_{\varepsilon}^p \Gamma(1-(p-1)d/\alpha) \gamma_{d,\alpha}^{(p)} t^{-(p+d/\alpha)} g(t,x)^p\\
&\quad \times \int_0^{t-\varepsilon}\frac{(t-s)^{p+d/\alpha}}{s^{(p-1)d/\alpha}}
E_{1-(p-1)d/\alpha, 1-(p-1)d/\alpha}
\left(c_{**}^{(p)} \Lambda_{d,\alpha}^{(p)}s^{1-(p-1)d/\alpha}\right)\,\d s.
\end{split}
\end{equation}
\end{lem}

\begin{proof}
By Lemma \ref{lem:g-comp}(1), (2) and Lemma \ref{lem:g-conv}(1),
\begin{equation*}
\begin{split}
Qu_0(t,x)=\int_{\R^d}q_t(x-y)u_0(y)\,\d y
&\ge C_{1,g}\int_{\R^d}g(t,x-y)u_0(y)\,\d y\\
&\ge C_{1,g}\kappa_{d,\alpha}^{-1} t^{d/\alpha}g(t,\sqrt{2}x) \int_{\R^d} g(t,\sqrt{2}y)u_0(y)\,\d y\\
&\ge \frac{C_{1,g}}{2^{d+\alpha}\kappa_{d,\alpha}} t^{d/\alpha}g(t,x)\int_{\R^d} g(t,y)u_0(y)\,\d y.
\end{split}
\end{equation*}
By  Lemma \ref{lem:g-comp}(4), we have for any  $\varepsilon\in (0,t]$ and $y\in \R^d$,
$t^{d/\alpha}g(t,y)\ge \varepsilon^{d/\alpha}g(\varepsilon,y)$,
which  implies that  for any $\varepsilon>0$, $t>0$ and $x\in \R^d$,
\begin{equation}\label{eq:q-lower}
Qu_0(t,x)\ge \frac{C_{1,g}\varepsilon^{d/\alpha}}{2^{d+\alpha}\kappa_{d,\alpha}} \int_{\R^d} g(\varepsilon,y)u_0(y)\,\d y \cdot g(t,x)
{\bf 1}_{[\varepsilon,\infty)}(t)
=C_{\varepsilon}g(t,x){\bf 1}_{[\varepsilon,\infty)}(t).
\end{equation}
Note that $C_{\varepsilon}$ is strictly positive because $u_0$ is nonnegative on $\R^d$
and strictly positive on a set of positive Lebesgue measure,
as well as $g(\varepsilon,y)>0$ for all $y\in \R^d$.
We also see by Lemma \ref{lem:g-comp}(1) that
\[
(\underbrace{q^p \star \dots \star q^p}_k)(s,y)
\ge C_{1,g}^{pk}(\underbrace{g^p \star \dots \star g^p}_k)(s,y).
\]
Therefore,
\begin{equation}\label{eq:iteration-3}
\begin{split}
&\sum_{k=1}^{\infty} \left(\frac{c_*^{(p)}}{4}\right)^k  (\underbrace{q^p\star \dots \star q^p}_k\star(Q u_0)^p)(t,x)\\
&=\sum_{k=1}^{\infty} \left(\frac{c_*^{(p)}}{4}\right)^k  \int_{(0,t]\times \R^d}Qu_0(t-s,x-y)^p(\underbrace{q^p\star \dots \star q^p}_k)(s,y)\,\d s\d y\\
&\ge \sum_{k=1}^{\infty} \left(\frac{c_*^{(p)}}{4}\right)^k
\int_{(0,t]\times \R^d}C_{\varepsilon}^p g(t-s,x-y)^p{\bf 1}_{[\varepsilon,\infty)}(t-s)
C_{1,g}^{pk}(\underbrace{g^p \star \dots \star g^p}_k)(s,y)\,\d s\d y\\
&=
C_{\varepsilon}^p\int_{(0,t-\varepsilon]\times \R^d} g(t-s,x-y)^p
\sum_{k=1}^{\infty} (c_{**}^{(p)})^k  (\underbrace{g^p \star \dots \star g^p}_k)(s,y)\,\d s\d y\\
&=
c_{**}^{(p)}C_{\varepsilon}^p\int_{(0,t-\varepsilon]\times \R^d} g(t-s,x-y)^p
K^{(p)}(c_{**}^{(p)};s,y)\,\d s\d y.
\end{split}
\end{equation}
See \eqref{eq:k-def} for the definition of $K^{(p)}(c_{**}^{(p)};s,y)$.
We here note that for any $s\in (0,t-\varepsilon]$ and $x\in \R^d$,  Lemma \ref{lem:g-conv}(3) and Lemma \ref{lem:g-comp}(3) yield
\begin{equation*}
\begin{split}
\{g(t-s,\cdot)^p*g(s,\cdot)^p\}(x)
&\ge \gamma_{d,\alpha}^{(p)} \frac{(t-s)^{d/\alpha}}{s^{(p-1)d/\alpha}t^{d/\alpha}}g(t-s,x)^p\\
&\ge  \gamma_{d,\alpha}^{(p)} \frac{(t-s)^{d/\alpha}}{s^{(p-1)d/\alpha}t^{d/\alpha}}\left(\frac{t-s}{t}g(t,x)\right)^p\\
&=\gamma_{d,\alpha}^{(p)} \frac{(t-s)^{p+d/\alpha}}{s^{(p-1)d/\alpha}t^{p+d/\alpha}}g(t,x)^p.
\end{split}
\end{equation*}
Then by Lemma \ref{prop:k-lower}, the last expression of \eqref{eq:iteration-3} is greater than
\begin{equation*}
\begin{split}
&c_{**}^{(p)}C_{\varepsilon}^p \Gamma(1-(p-1)d/\alpha)\\
&\times \int_{(0,t-\varepsilon]\times \R^d} g(t-s,x-y)^p g(s,y)^p
E_{1-(p-1)d/\alpha, 1-(p-1)d/\alpha}
\left(c_{**}^{(p)}\Lambda_{d,\alpha}^{(p)}s^{1-(p-1)d/\alpha}\right)\,\d s\d y\\
&=c_{**}^{(p)}C_{\varepsilon}^p \Gamma(1-(p-1)d/\alpha) \\
&\times \int_0^{t-\varepsilon} \{g(t-s,\cdot)^p*g(s,\cdot)^p\}(x)
E_{1-(p-1)d/\alpha, 1-(p-1)d/\alpha}
\left(c_{**}^{(p)}\Lambda_{d,\alpha}^{(p)}s^{1-(p-1)d/\alpha}\right)\,\d s \\
&\ge  c_{**}^{(p)}C_{\varepsilon}^p \Gamma(1-(p-1)d/\alpha) \gamma_{d,\alpha}^{(p)}t^{-(p+d/\alpha)} g(t,x)^p \\
&\times \int_0^{t-\varepsilon}\frac{(t-s)^{p+d/\alpha}}{s^{(p-1)d/\alpha}}
E_{1-(p-1)d/\alpha, 1-(p-1)d/\alpha}
\left(c_{**}^{(p)}\Lambda_{d,\alpha}^{(p)}s^{1-(p-1)d/\alpha}\right)\,\d s.
\end{split}
\end{equation*}
Combining this with \eqref{eq:iteration-2},
we get \eqref{eq:iteration-4}.
\end{proof}

\begin{proof}[Proof of Proposition {\rm \ref{prop:p=2}}]
We first discuss the lower bound of \eqref{eq:iteration-4}.
By definition, the Mittag-Leffler function $E_{1-(p-1)d/\alpha, 1-(p-1)d/\alpha}(z)$ is increasing in $z>0$  (see \eqref{eq:ML}),
and so
\begin{equation}\label{eq:iteration-5}
\begin{split}
&\int_0^{t-\varepsilon}\frac{(t-s)^{p+d/\alpha}}{s^{(p-1)d/\alpha}}
E_{1-(p-1)d/\alpha, 1-(p-1)d/\alpha}
\left(c_{**}^{(p)}\Lambda_{d,\alpha}^{(p)}s^{1-(p-1)d/\alpha}\right)\,\d s\\
&\ge \int_{t-2\varepsilon}^{t-\varepsilon}\frac{(t-s)^{p+d/\alpha}}{s^{(p-1)d/\alpha}}
E_{1-(p-1)d/\alpha, 1-(p-1)d/\alpha}
\left(c_{**}^{(p)}\Lambda_{d,\alpha}^{(p)}s^{1-(p-1)d/\alpha}\right)\,\d s\\
&\ge  \int_{t-2\varepsilon}^{t-\varepsilon}\frac{(t-s)^{p+d/\alpha}}{s^{(p-1)d/\alpha}}\,\d s
\cdot E_{1-(p-1)d/\alpha, 1-(p-1)d/\alpha}\left(c_{**}^{(p)}\Lambda_{d,\alpha}^{(p)}(t-2\varepsilon)^{1-(p-1)d/\alpha}\right)\\
&\ge \varepsilon^{p+1+d/\alpha}(t-\varepsilon)^{-(p-1)d/\alpha}
E_{1-(p-1)d/\alpha, 1-(p-1)d/\alpha}\left(c_{**}^{(p)}\Lambda_{d,\alpha}^{(p)}(t-2\varepsilon)^{1-(p-1)d/\alpha}\right).
\end{split}
\end{equation}
Since it follows by \eqref{eq:ML-1} that
as $z\rightarrow\infty$,
\[
E_{1-(p-1)d/\alpha,1-(p-1)d/\alpha}(z)\sim (1-(p-1)d/\alpha)^{-1}z^{(p-1)d/\{\alpha(1-(p-1)d/\alpha)\}}
\exp\left(z^{1/(1-(p-1)d/\alpha)}\right),
\]
we obtain as $t\rightarrow\infty$,
\begin{equation*}
\begin{split}
&(t-\varepsilon)^{-(p-1)d/\alpha}E_{1-(p-1)d/\alpha, 1-(p-1)d/\alpha}
\left(c_{**}^{(p)}\Lambda_{d,\alpha}^{(p)}(t-2\varepsilon)^{1-(p-1)d/\alpha}\right)\\
&\sim t^{-(p-1)d/\alpha} (1-(p-1)d/\alpha)^{-1}(c_{**}^{(p)}\Lambda_{d,\alpha}^{(p)})^{(p-1)d/\{\alpha(1-(p-1)d/\alpha)\}}\\
&\quad \times t^{(p-1)d/\alpha}
\exp\left((c_{**}^{(p)}\Lambda_{d,\alpha}^{(p)})^{1/(1-(p-1)d/\alpha)}(t-2\varepsilon)\right)\\
&=(1-(p-1)d/\alpha)^{-1}(c_{**}^{(p)}\Lambda_{d,\alpha}^{(p)})^{(p-1)d/\{\alpha(1-(p-1)d/\alpha)\}}\\
&\quad \times
\exp\left(-2(c_{**}^{(p)}\Lambda_{d,\alpha}^{(p)})^{1/(1-(p-1)d/\alpha)} \varepsilon\right)
\exp\left((c_{**}^{(p)}\Lambda_{d,\alpha}^{(p)})^{1/(1-(p-1)d/\alpha)}t\right).
\end{split}
\end{equation*}
Therefore, if we let
\[
C_{\varepsilon,2}^{(p)}
=\frac{\varepsilon^{p+1+d/\alpha} (c_{**}^{(p)}\Lambda_{d,\alpha}^{(p)})^{(p-1)d/\{\alpha(1-(p-1)d/\alpha)\}}}{2(1-(p-1)d/\alpha)}
\exp\left(-2(c_{**}^{(p)}\Lambda_{d,\alpha}^{(p)})^{1/(1-(p-1)d/\alpha)} \varepsilon\right),
\]
then by \eqref{eq:iteration-5}, there exists $T>0$ such that for all $t\ge T$,
\begin{equation*}
\begin{split}
&\int_0^{t-\varepsilon}\frac{(t-s)^{p+d/\alpha}}{s^{(p-1)d/\alpha}}
E_{1-(p-1)d/\alpha, 1-(p-1)d/\alpha}\left(c_{**}^{(p)}\Lambda_{d,\alpha}^{(p)}s^{1-(p-1)d/\alpha}\right)\,\d s\\
&\ge C_{\varepsilon,2}^{(p)}\exp\left((c_{**}^{(p)}\Lambda_{d,\alpha}^{(p)})^{1/(1-(p-1)d/\alpha)}t\right).
\end{split}
\end{equation*}
In particular, we have by \eqref{eq:iteration-4},
\[
E[|X(t,x)|^p]\\
\ge c_{**}^{(p)}C_{\varepsilon}^p C_{\varepsilon,2}^{(p)} \Gamma(1-(p-1)d/\alpha) \gamma_{d,\alpha}^{(p)}
t^{-(p+d/\alpha)} g(t,x)^p
\exp\left((c_{**}^{(p)}\Lambda_{d,\alpha}^{(p)})^{1/(1-(p-1)d/\alpha)}t\right).
\]
Moreover, since $\sup_{|x|\ge e^{\eta t}}g(t,x)=g(t,e^{\eta t})$, we get for any $\eta>0$,
\begin{equation}\label{eq:iteration-6}
\begin{split}
&\sup_{|x|\ge e^{\eta t}}E[|X(t,x)|^p]\\
&\ge c_{**}^{(p)}C_{\varepsilon}^p C_{\varepsilon,2}^{(p)}\Gamma(1-(p-1)d/\alpha) \gamma_{d,\alpha}^{(p)} t^{-(p+d/\alpha)} g(t,e^{\eta t})^p
\exp\left((c_{**}^{(p)}\Lambda_{d,\alpha}^{(p)})^{1/(1-(p-1)d/\alpha)}t\right).
\end{split}
\end{equation}

Since it follows by \eqref{eq:g-def} that
\[
\lim_{t\rightarrow\infty}\frac{1}{t}\log (g(t,e^{\eta t})^p)=-p(d+\alpha)\eta,
\]
we have by \eqref{eq:iteration-6},
\[
\liminf_{t\rightarrow\infty}\frac{1}{t}\log \left(\sup_{|x|\ge e^{\eta t}}E[|X(t,x)|^p]\right)\\
\ge (c_{**}^{(p)}\Lambda_{d,\alpha}^{(p)})^{1/(1-(p-1)d/\alpha)}-p(d+\alpha)\eta,
\]
which yields
\[
\underline{\lambda}(p)\ge \frac{(c_{**}^{(p)}\Lambda_{d,\alpha}^{(p)})^{1/(1-(p-1)d/\alpha)}}{p(d+\alpha)}.
\]
The proof is complete.
\end{proof}

\begin{rem}\label{rem:h-moment-1}\rm
Under the setting of Proposition \ref{prop:p=2},
we can verify the mild solution $X(t,x)$ has no higher order moment.
In fact, by the definition of the Mittag-Leffler function in \eqref{eq:ML},
we see that for any constants $a>0$, $b>0$ and  $z\ge 0$, $E_{a,b}(z)\ge 1/\Gamma(b)$.
Then by \eqref{eq:iteration-4}, we have for all  $t>0$, $x\in \R^d$ and $\varepsilon \in (0,t)$,
\[
E[|X(t,x)|^p]
\ge c_{**}^{(p)}C_{\varepsilon}^p \gamma_{d,\alpha}^{(p)} t^{-(p+d/\alpha)} g(t,x)^p\\
\int_0^{t-\varepsilon}\frac{(t-s)^{p+d/\alpha}}{s^{(p-1)d/\alpha}}\,\d s\rightarrow \infty,
\quad p\rightarrow 1+\frac{\alpha}{d}.
\]
This implies that $E[|X(t,x)|^{1+\alpha/d}]=\infty$ for all  $t>0$ and $x\in \R^d$.
\end{rem}

\begin{proof}[Proof of Theorem {\rm \ref{thm1}(ii)}]
The assertion follows by Propositions \ref{prop:exp-upper} and \ref{prop:p=2}.
\end{proof}

\section{Lower bounds for the growth indices of subexponential type:
$p\in (1,2)$}\label{section5}
In this section, we prove a lower bound
for the growth indices of subexponential type on the $p$th moment of the mild solution to \eqref{eq:fshe}
for $p\in (1,2)$.
By comparison with the case that $p\ge 2$,
the difficulty lies in the lack of the maximal inequality for purely discontinuous martingales
when $p\in (1,2)$.
More precisely, the analogy of \eqref{eq:BDG-p=2} is invalid in general for the
$p$th
moment of the mild solution (in particular with $p\in (1,2)$).
Hence we could not directly extend the argument in Proposition \ref{prop:p=2} to $p\in (1,2)$.

On the other hand,
when $\alpha=2$, Chong-Kevei \cite[Lemma 3.4]{CK19} proved a lower bound
on the $p$th moment of the Poisson stochastic integral.
Using this and the renewal inequalities,
they further established the positivity of the growth indices of linear type
(\cite[Theorem 3.6]{CK19}).
In their argument, the linear growth condition is utilized in \cite[Proof of Theorem 3.6]{CK19}.
However, since our principal
part
in \eqref{eq:fshe} is the fractional 
Laplacian operator,
the exponential growth condition is expected and so the renewal 
inequalities are not applicable for \eqref{eq:fshe}
with $\alpha\in (0,2)$.

Our approach here is to replace \eqref{eq:BDG-p=2} by \cite[Lemma 3.4]{CK19}
(see also Lemma \ref{lem:PG-ineq} for the statement) when $p\in (1,2)$,
and then to follow the moment calculus as in the proof of Proposition \ref{prop:p=2}.
Moreover, we also introduce the new measure $\mu_0$ defined by \eqref{eq:mu0} in order to make use of the semigroup property of the heat kernel $q_t(x)$.

\begin{prop}\label{thm:p<2}
Assume that, as in Proposition {\rm \ref{prop:p=2}}, all of the following conditions are satisfied{\rm :}
\begin{itemize}
\item $b=0$.
\item
$u_0(x)\ge 0$ for all $x\in \mathbb R^d$, and it is strictly positive on a set of positive Lebesgue measure.
\item
$L_{\sigma,0}:=\inf_{w\in {\mathbb R}\setminus \{0\}}|\sigma(w)|/|w|>0$.
\end{itemize}
If $\alpha>d=1$, then
for any $p\in (1,2)$,
there exists $\eta_*>0$ such that for any $\eta\in (0,\eta_*)$,
\[
\liminf_{t\rightarrow\infty}\frac{1}{t^{r_*}}\log \left(\sup_{|x|\ge e^{\eta t^{r_*}}}E[|X(t,x)|^p]\right)>0
\]
with $r_*=p(1-d/\alpha)/\{2(1-(p-1)d/\alpha)\}$.
\end{prop}
\begin{rem}\rm If $p=2$, then $r_{*}=1$ and Proposition \ref{thm:p<2} is reduced into Proposition \ref{prop:p=2}.
On the other hand, if $1<p<2$, then $0<r_{*}<1$.
In particular, our bound in Proposition \ref{thm:p<2} does not match the upper bound in Proposition \ref{prop:exp-upper}.
\end{rem}

In what follows, we keep the same setting as in Proposition \ref{thm:p<2}. We will follow the proof of Proposition \ref{prop:p=2} and use some notations herein.
Let $p\in (1,2)$. Then, by  \eqref{eq:p-moment} and \eqref{eq:p-moment-2},
there exists a constant $c_0>0$ such that for any $t>0$ and $x\in \R^d$,
\begin{equation}\label{eq:pth-moment-2}
\begin{split}
&E[|X(t,x)|^p]\\
&\ge \frac{1}{4}(Qu_0(t,x))^p
+c_0E\left[\left(\int_{[0,t)\times \R^d\times \R}q_{t-s}(x-y)^2\sigma(X(s,y))^2z^2\,\mu(\d s, \d y,\d z)\right)^{p/2}\right].
\end{split}
\end{equation}
Fix $t>0$, $x\in \R^d$ and $\delta>0$.
We then define the measures $\mu_0$ and $\nu_0$ on $[0,\infty)\times \R^d\times \R$, respectively,  by
\begin{equation}\label{eq:mu0}
\mu_0(\d s, \d y, \d z)={\bf 1}_{\{|z|>\delta\}}{\bf 1}_{\{0\le s<t\}}\frac{q_{t-s}(x-y)}{q_{t-s}(0)}\,\mu(\d s,\d y, \d z)
\end{equation}
and
\[
\nu_0(\d s, \d y, \d z)={\bf 1}_{\{|z|>\delta\}}{\bf 1}_{\{0\le s<t\}}\frac{q_{t-s}(x-y)}{q_{t-s}(0)}\,\nu(\d s,\d y, \d z).
\]
Since $\displaystyle\int_{\R^d}q_{t-s}(x-y)\,\d y=1$ and $q_s(0)\asymp s^{-d/\alpha}$, we have
\begin{equation*}
\begin{split}
\nu_0([0,\infty)\times \R^d\times \R)
&=\lambda([-\delta,\delta]^c)\int_0^t \left(\int_{\R^d}\frac{q_{t-s}(x-y)}{q_{t-s}(0)}\,\d y\right) \, \d s\\
&=\lambda([-\delta,\delta]^c)\int_0^t \frac{1}{q_{t-s}(0)}\, \d s=\lambda([-\delta,\delta]^c)\int_0^t \frac{1}{q_s(0)}\, \d s
\asymp t^{1+d/\alpha}.
\end{split}
\end{equation*}

Note also that $q_{t-s}(x-y)/q_{t-s}(0)\le 1$.
Then, by the Burkholder-Davis-Gundy inequality,
\begin{equation*}
\begin{split}
&E\left[\left(\int_{[0,t)\times \R^d\times \R}q_{t-s}(x-y)^2\sigma(X(s,y))^2z^2\,\mu(\d s, \d y,\d z)\right)^{p/2}\right]\\
&\ge E\left[\left(\int_{[0,t)\times \R^d\times \R}q_{t-s}(x-y)^2\sigma(X(s,y))^2z^2\,\mu_0(\d s, \d y,\d z)\right)^{p/2}\right]\\
&\asymp  E\left[\left|\int_{[0,t)\times \R^d\times \R}q_{t-s}(x-y) \sigma(X(s,y)) z\,\mu_0(\d s, \d y,\d z)\right|^p\right].
\end{split}
\end{equation*}
Applying  Lemma \ref{lem:PG-ineq} to the last expression above, we further obtain for all large $t>0$,
\begin{equation*}
\begin{split}
&E\left[\left|\int_{[0,t)\times \R^d\times \R}q_{t-s}(x-y) \sigma(X(s,y)) z\,\mu_0(\d s, \d y,\d z)\right|^p\right]\\
&\gtrsim \frac{\int_{[0,t)\times \R^d\times \R} q_{t-s}(x-y)^p|z|^p E[|\sigma(X(s,y))|^p]\nu_0(\d s, \d y,\d z)}{
\left\{1\vee \nu_0([0,t)\times \R^d\times \R)\right\}^{1-p/2}}\\
&\asymp \frac{1}{t^{(1+d/\alpha)(1-p/2)}} \int_{|z|>\delta}|z|^p\,\lambda(\d z)
\int_{[0,t)\times \R^d}\frac{q_{t-s}(x-y)^{p+1}}{q_{t-s}(0)}E[|\sigma(X(s,y))|^p]\,\d s\d y\\
&=\frac{\sigma_{\lambda}^{(p)}}{t^{(1+d/\alpha)(1-p/2)}}
\int_{[0,t)\times \R^d}q^{(p)}(t-s,x-y)E[|\sigma(X(s,y))|^p]\,\d s\d y
\end{split}
\end{equation*}
with $\sigma_{\lambda}^{(p)}=\int_{|z|>\delta}|z|^p\,\lambda(\d z)$ and $q^{(p)}(t,x)=q_t(x)^{p+1}/q_t(0)$.
Then, by the two inequalities above and \eqref{eq:pth-moment-2},
there exists $c_1>0$ such that
\begin{equation}\label{eq:p-iteration}
\begin{split}
&E[|X(t,x)|^p]\\
&\ge \frac{1}{4}(Qu_0(t,x))^p
+\frac{c_1\sigma_{\lambda}^{(p)}}{t^{(1+d/\alpha)(1-p/2)}}
\int_{[0,t)\times \R^d}q^{(p)}(t-s,x-y)E[|\sigma(X(s,y))|^p]\,\d s\d y.
\end{split}
\end{equation}

Let
$c_*^{(p)}=c_1 \sigma_{\lambda}^{(p)} L_{\sigma,0}^p$.
We then have
\begin{lem}\label{lem:p<2}
Under the same setting as in Proposition {\rm \ref{thm:p<2}}, for any $t>0$ and $x\in \R^d$,
\begin{equation}\label{eq:p-iteration-2}
E[|X(t,x)|^p]
\ge \sum_{k=1}^{\infty} \left(\frac{c_*^{(p)}}{4 t^{(1+d/\alpha)(1-p/2)}}\right)^k  (\underbrace{q^{(p)}\star \dots \star q^{(p)}}_k\star(Q u_0)^p)(t,x).
\end{equation}
\end{lem}
\begin{proof}
Since $L_{\sigma,0}>0$, we have by \eqref{eq:p-iteration},
\begin{equation}\label{eq:p-iteration-1}
E[|X(t,x)|^p]
\ge\frac{1}{4}(Qu_0(t,x))^p
+\frac{c_*^{(p)}}{t^{(1+d/\alpha)(1-p/2)}}\int_0^t\left(\int_{\R^d}q_{t-s}^{(p)}(x-y)E\left[|X(s,y)|^p\right]\,\d y\right)\d s,
\end{equation}
which implies that
\begin{equation*}
\begin{split}
&\int_0^t\left(\int_{\R^d}q_{t-s}^{(p)}(x-y)E\left[|X(s,y)|^p\right]\,\d y\right)\d s\\
&\ge \int_0^t\Biggl[\int_{\R^d}q_{t-s}^{(p)}(x-y)\\
&\qquad
\times \left\{\frac{1}{4}(Qu_0(s,y))^p
+\frac{c_*^{(p)}}{s^{(1+d/\alpha)(1-p/2)}}\left(\int_0^s\int_{\R^d}q_{s-u}^{(p)}(y-z)E\left[|X(u,z)|^p\right]\,\d z\d u\right)\right\}\,\d y\Biggr]\,\d s\\
&=\frac{1}{4}\int_0^t \left(\int_{\R^d}q_{t-s}^{(p)}(x-y)(Qu_0(s,y))^p\,\d y\right)\,\d s \\
&\quad +c_*^{(p)} \int_0^t\int_{\R^d} \frac{1}{s^{(1+d/\alpha)(1-p/2)}}q_{t-s}^{(p)}(x-y)  \left(\int_0^s \int_{\R^d}q_{s-u}^{(p)}(y-z)E[|X(u,z)|^p]\,\d z\,\d u\right)\,\d y\,\d s.
\end{split}
\end{equation*}
Then
\[
\int_0^t \left(\int_{\R^d}q_{t-s}^{(p)}(x-y)(Qu_0(s,y))^p\,\d y\right)\,\d s
=\{q^{(p)}\star (Qu_0)^p\}(t,x),
\]
and the Fubini theorem yields
\begin{equation*}
\begin{split}
&\int_0^t\int_{\R^d} \frac{1}{s^{(1+d/\alpha)(1-p/2)}} q_{t-s}^{(p)}(x-y) \left(\int_0^s \int_{\R^d}q_{s-u}^{(p)}(y-z)E[|X(u,z)|^p]\,\d z\,\d u\right)\,\d y\,\d s \\
&=\int_0^t \int_{\R^d} \left(\int_0^{t-u}\int_{\R^d} \frac{1}{(u+v)^{(1+d/\alpha)(1-p/2)}} q_{t-u-v}^{(p)}(x-y) q_{v}^{(p)}(y-z)\,\d y\,\d v\right)E[|X(u,z)|^p]\,\d z\,\d u\\
&\ge \frac{1}{t^{(1+d/\alpha)(1-p/2)}}\int_0^t \int_{\R^d} \left(\int_0^{t-u}\int_{\R^d}  q_{t-u-v}^{(p)}(x-z-y) q_{v}^{(p)}(y)\,\d y\,\d v\right)E[|X(u,z)|^p]\,\d z\,\d u\\
&= \frac{1}{t^{(1+d/\alpha)(1-p/2)}} \int_0^t \int_{\R^d} (q^{(p)}\star q^{(p)})(t-u,x-z)E[|X(u,z)|^p]\,\d z\,\d u.
\end{split}
\end{equation*}
We thus get
\begin{equation*}
\begin{split}
&\int_0^t\left(\int_{\R^d}q_{t-s}^{(p)}(x-y)E\left[|X(s,y)|^p\right]\,\d y\right)\d s\\
&\ge \frac{1}{4}\{q^{(p)}\star (Qu_0)^p\}(t,x)+\frac{c_*^{(p)}}{t^{(1+d/\alpha)(1-p/2)}}\int_0^t \int_{\R^d} (q^{(p)}\star q^{(p)})(t-u,x-z)E[|X(u,z)|^p]\,\d z\,\d u.
\end{split}
\end{equation*}
Then, by \eqref{eq:p-iteration-1},
\begin{equation*}
\begin{split}
E[|X(t,x)|^p]
&\ge \frac{1}{4}(Qu_0(t,x))^p
+\frac{c_*^{(p)}}{4t^{(1+d/\alpha)(1-p/2)}}\{q^{(p)}\star (Qu_0)^p\}(t,x)\\
&\quad +\left(\frac{c_*^{(p)}}{t^{(1+d/\alpha)(1-p/2)}}\right)^2\int_0^t \int_{\R^d} (q^{(p)}\star q^{(p)})(t-u,x-z)E[|X(u,z)|^p]\,\d z\,\d u.
\end{split}
\end{equation*}

Repeating the same iterative operation as above, we obtain
\begin{equation*}
\begin{split}
E[|X(t,x)|^p]
&\ge \frac{1}{4}(Qu_0(t,x))^p
+\sum_{k=1}^n \left(\frac{c_*^{(p)}}{4t^{(1+d/\alpha)(1-p/2)}}\right)^k  (\underbrace{q^{(p)}\star \dots \star q^{(p)}}_k\star(Q u_0)^p)(t,x)\\
&\quad + \left(\frac{c_*^{(p)}}{ t^{(1+d/\alpha)(1-p/2)}}\right)^{n+1}
\int_0^t \int_{\R^d} (\underbrace{q^{(p)}\star \cdots\star q^{(p)}}_{n+1})(t-u,x-z)E[|X(u,z)|^p]\,\d z\,\d u\\
&\ge \sum_{k=1}^n \left(\frac{c_*^{(p)}}{4 t^{(1+d/\alpha)(1-p/2)}}\right)^k  (\underbrace{q^{(p)}\star \dots \star q^{(p)}}_k\star(Q u_0)^p)(t,x).
\end{split}
\end{equation*}
Letting $n\rightarrow\infty$, we arrive at \eqref{eq:p-iteration-2}.
\end{proof}

Using Lemma \ref{lem:p<2}, we further calculate the lower bound of $E[|X(t,x)|^p]$.
In what follows, let
$\gamma_{d,_\alpha}^{(p+1)}$ and $\Lambda_{d,\alpha}^{(p+1)}$
be as in Lemma \ref{lem:g-conv},
and let $\Theta_{d,\alpha}^{(p)}$ be as in Lemma \ref{lem:p-g-conv}.
Let
$C_{\varepsilon}$ be as in \eqref{eq:c_e}
and
$c_{**}^{(p)}=c_*^{(p)}C_{1,g}^{p+1}/(4C_{2,g})$,
where $\kappa_{d,\alpha}$ is defined in \eqref{eq:g-def}.
\begin{lem}
Under the same setting as in Proposition {\rm \ref{thm:p<2}}, for any $t>0$ and $x\in \R^d$,
\begin{equation}\label{eq:p-iteration-0}
\begin{split}
E[|X(t,x)|^p]
&\ge \frac{c_{**}^{(p)}C_{\varepsilon}^p\gamma_{d,\alpha}^{(p+1)}}{\kappa_{d,\alpha} t^{(1+d/\alpha)(1-p/2)+p+1+2d/\alpha}}
\Gamma(1-(p-1)d/\alpha) g^{(p)}(t,x)\\
&\quad \times \int_0^{t-\varepsilon}\frac{(t-s)^{p+1+2d/\alpha}}{s^{(p-1)d/\alpha}}
E_{1-(p-1)d/\alpha, 1-(p-1)d/\alpha}
\left(\frac{c_{**}\Theta_{d,\alpha}^{(p)}s^{1-(p-1)d/\alpha}}{t^{(1+d/\alpha)(1-p/2)}}\right)\,\d s.
\end{split}
\end{equation}
\end{lem}

\begin{proof}
Define $g^{(p)}(t,x)=g^{p+1}(t,x)/g(t,0)$.
Then, since  Lemma \ref{lem:g-comp}(1) yields
\[
(\underbrace{q^{(p)}\star \dots \star q^{(p)}}_k)(s,y)
\ge \left(\frac{C_{1,g}^{p+1}}{C_{2,g}}\right)^k(\underbrace{g^{(p)}\star \dots \star g^{(p)}}_k)(s,y),
\]
we have by Lemma \ref{lem:p<2}
and \eqref{eq:q-lower},
\begin{equation}\label{eq:p-iteration-3}
\begin{split}
&E[|X(t,x)|^p]\\
&\ge \sum_{k=1}^{\infty} \left(\frac{c_*^{(p)}}{4 t^{(1+d/\alpha)(1-p/2)}}\right)^k  (\underbrace{q^{(p)}\star \dots \star q^{(p)}}_k\star(Q u_0)^p)(t,x)\\
&=\sum_{k=1}^{\infty} \left(\frac{c_*^{(p)}}{4 t^{(1+d/\alpha)(1-p/2)}}\right)^k  \int_{(0,t]\times \R^d}Qu_0(t-s,x-y)^p(\underbrace{q^{(p)}\star \dots \star q^{(p)}}_k)(s,y)\,\d s\d y\\
&\ge \sum_{k=1}^{\infty} \left(\frac{c_*^{(p)}}{4 t^{(1+d/\alpha)(1-p/2)}}\right)^k \\
&\quad \times \int_{(0,t]\times \R^d}C_{\varepsilon}^p g(t-s,x-y)^p{\bf 1}_{[\varepsilon,\infty)}(t-s) \left(\frac{C_{1,g}^{p+1}}{C_{2,g}}\right)^k
(\underbrace{g^{(p)}\star \dots \star g^{(p)}}_k)(s,y)\,\d s\d y\\
&=
\frac{c_{**}^{(p)} C_{\varepsilon}^p }{t^{(1+d/\alpha)(1-p/2)}} \int_{(0,t-\varepsilon]\times \R^d} g(t-s,x-y)^p
K_1^{(p)}\left(\frac{c_{**}^{(p)}}{t^{(1+d/\alpha)(1-p/2)}} ;s,y\right)\,\d s\d y,
\end{split}
\end{equation}
where $K_1^{(p)}$ is defined by \eqref{eq:k1-def}.
Moreover, by Lemma \ref{prop:p-k-lower},
the last expression above is greater than
\begin{equation}\label{eq:p-iteration-4}
\begin{split}
&\frac{c_{**}^{(p)} C_{\varepsilon}^p }{t^{(1+d/\alpha)(1-p/2)}}  \Gamma(1-(p-1)d/\alpha)\\
&\times \int_{(0,t-\varepsilon]\times \R^d} g(t-s,x-y)^p g^{(p)}(s,y)
E_{1-(p-1)d/\alpha, 1-(p-1)d/\alpha}
\left(\frac{c_{**}^{(p)} \Theta_{d,\alpha}^{(p)}s^{1-(p-1)d/\alpha}}{t^{(1+d/\alpha)(1-p/2)}}\right)\,\d s\d y\\
&=\frac{c_{**}^{(p)}  C_{\varepsilon}^p  }{t^{(1+d/\alpha)(1-p/2)}}  \Gamma(1-(p-1)d/\alpha)\\
&\quad \times \int_0^{t-\varepsilon} \left(\int_{\R^d}g(t-s,x-y)^p g^{(p)}(s,y)\,\d y\right)
E_{1-(p-1)d/\alpha, 1-(p-1)d/\alpha}
\left(\frac{c_{**}^{(p)} \Theta_{d,\alpha}^{(p)}s^{1-(p-1)d/\alpha}}{t^{(1+d/\alpha)(1-p/2)}}\right)\,\d s.
\end{split}
\end{equation}

On the other hand, since $g(t-s,x-y)\le g(t-s,0)=\kappa_{d,\alpha}/(t-s)^{d/\alpha}$,
we have
\begin{equation*}
\begin{split}
\int_{\R^d}g(t-s,x-y)^p g(s,y)^{(p)}\,\d y
&=\int_{\R^d}\frac{g(t-s,x-y)^{p+1}}{g(t-s,x-y)}\frac{g(s,y)^{p+1}}{g(s,0)}\,\d y\\
&\ge \frac{1}{g(t-s,0)g(s,0)} \int_{\R^d} g(t-s,x-y)^{p+1}g(s,y)^{p+1}\,\d y\\
&=\frac{(t-s)^{d/\alpha}s^{d/\alpha}}{\kappa_{d,\alpha}^2}\int_{\R^d} g(t-s,x-y)^{p+1}g(s,y)^{p+1}\,\d y.
\end{split}
\end{equation*}
Then for any $s\in (0,t-\varepsilon]$ and $x\in \R^d$,  Lemma \ref{lem:g-conv}(3) and Lemma \ref{lem:g-comp}(3) yield
\begin{equation*}
\begin{split}
\int_{\R^d} g(t-s,x-y)^{p+1}g(s,y)^{p+1}\,\d y
&\ge \gamma_{d,\alpha}^{(p+1)}\frac{(t-s)^{d/\alpha}}{s^{pd/\alpha}t^{d/\alpha}}g(t-s,x)^{p+1}\\
&\ge  \gamma_{d,\alpha}^{(p+1)}\frac{(t-s)^{d/\alpha}}{s^{pd/\alpha}t^{d/\alpha}}\left(\frac{t-s}{t}g(t,x)\right)^{p+1}\\
&=\gamma_{d,\alpha}^{(p+1)}\frac{(t-s)^{p+1+d/\alpha}}{s^{pd/\alpha}t^{p+1+d/\alpha}}g(t,x)^{p+1}.
\end{split}
\end{equation*}
By the two inequalities above, we obtain
\begin{equation*}
\begin{split}
\int_{\R^d}g(t-s,x-y)^p g^{(p)}(s,y)\,\d y
&\ge
\frac{\gamma_{d,\alpha}^{(p+1)}}{\kappa_{d,\alpha}^2}
\frac{(t-s)^{p+1+2d/\alpha}}{s^{(p-1)d/\alpha}t^{p+1+d/\alpha}}g(t,x)^{p+1}\\
&=\frac{\gamma_{d,\alpha}^{(p+1)}}{\kappa_{d,\alpha}}
\frac{(t-s)^{p+1+2d/\alpha}}{s^{(p-1)d/\alpha}t^{p+1+2d/\alpha}}g^{(p)}(t,x).
\end{split}
\end{equation*}
Combining this with \eqref{eq:p-iteration-3} and  \eqref{eq:p-iteration-4},
we get \eqref{eq:p-iteration-0}.
\end{proof}

\begin{proof}[Proof of Proposition {\rm \ref{thm:p<2}}]
We first discuss the lower bound of \eqref{eq:p-iteration-0}.
By definition, the Mittag-Leffler function $E_{1-(p-1)d/\alpha,1-(p-1)d/\alpha}(z)$ is increasing in $z>0$ (see \eqref{eq:ML}), and so
\begin{equation}\label{eq:p-iteration-5}
\begin{split}
&\int_0^{t-\varepsilon}\frac{(t-s)^{p+1+2d/\alpha}}{s^{(p-1)d/\alpha}}
E_{1-(p-1)d/\alpha, 1-(p-1)d/\alpha}
\left(\frac{c_{**}^{(p)} \Theta_{d,\alpha}^{(p)}s^{1-(p-1)d/\alpha}}{t^{(1+d/\alpha)(1-p/2)}}\right)\,\d s\\
&\ge \int_{t-2\varepsilon}^{t-\varepsilon}
\frac{(t-s)^{p+1+2d/\alpha}}{s^{(p-1)d/\alpha}}
E_{1-(p-1)d/\alpha, 1-(p-1)d/\alpha}
\left(\frac{c_{**}^{(p)} \Theta_{d,\alpha}^{(p)}s^{1-(p-1)d/\alpha}}{t^{(1+d/\alpha)(1-p/2)}}\right)\,\d s\\
&\ge \int_{t-2\varepsilon}^{t-\varepsilon}
\frac{(t-s)^{p+1+2d/\alpha}}{s^{(p-1)d/\alpha}}\,\d s
\cdot E_{1-(p-1)d/\alpha, 1-(p-1)d/\alpha}
\left(\frac{c_{**}^{(p)} \Theta_{d,\alpha}^{(p)}(t-2\varepsilon)^{1-(p-1)d/\alpha}}{t^{(1+d/\alpha)(1-p/2)}}\right)\\
&\ge \varepsilon^{p+2+2d/\alpha}(t-\varepsilon)^{-(p-1)d/\alpha}
E_{1-(p-1)d/\alpha, 1-(p-1)d/\alpha}
\left(\frac{c_{**}^{(p)} \Theta_{d,\alpha}^{(p)}(t-2\varepsilon)^{1-(p-1)d/\alpha}}{t^{(1+d/\alpha)(1-p/2)}}\right).
\end{split}
\end{equation}
Since it follows by \eqref{eq:ML-1} that as $z\rightarrow\infty$,
\[
E_{1-(p-1)d/\alpha,1-(p-1)d/\alpha}(z)\sim (1-(p-1)d/\alpha)^{-1}z^{(p-1)d/\{\alpha(1-(p-1)d/\alpha)\}}\exp\left(z^{1/(1-(p-1)d/\alpha)}\right),
\]
we obtain as $t\rightarrow\infty$,
\begin{equation*}
\begin{split}
&(t-\varepsilon)^{-(p-1)d/\alpha}
E_{1-(p-1)d/\alpha, 1-(p-1)d/\alpha}
\left(\frac{c_{**}^{(p)} \Theta_{d,\alpha}^{(p)}(t-2\varepsilon)^{1-(p-1)d/\alpha}}{t^{(1+d/\alpha)(1-p/2)}}\right)\\
&\sim t^{-(p-1)d/\alpha} (1-(p-1)d/\alpha)^{-1}(c_{**}^{(p)} \Theta_{d,\alpha}^{(p)})^{(p-1)d/\{\alpha(1-(p-1)d/\alpha)\}}\\
&\quad \times t^{(p-1)d/\alpha-(1+d/\alpha)(1-p/2)(p-1)d/\{\alpha(1-(p-1)d/\alpha)\}}
\exp\left(\frac{(c_{**}^{(p)} \Theta^{(p)}_{d,\alpha})^{1/(1-(p-1)d/\alpha)}(t-2\varepsilon)}{t^{(1+d/\alpha)(1-p/2)/(1-(p-1)d/\alpha)}}\right)\\
&=(1-(p-1)d/\alpha)^{-1}(c_{**}^{(p)} \Theta_{d,\alpha}^{(p)})^{(p-1)d/\{\alpha(1-(p-1)d/\alpha)\}}
\exp\left(-\frac{2(c_{**}^{(p)} \Theta^{(p)}_{d,\alpha})^{1/(1-(p-1)d/\alpha)}\varepsilon}{t^{(1+d/\alpha)(1-p/2)/(1-(p-1)d/\alpha)}}\right)\\
&\quad \times t^{-(1+d/\alpha)(1-p/2)(p-1)d/\{\alpha(1-(p-1)d/\alpha)\}}
\exp\left(\frac{(c_{**}^{(p)} \Theta^{(p)}_{d,\alpha})^{1/(1-(p-1)d/\alpha)}t}{t^{(1+d/\alpha)(1-p/2)/(1-(p-1)d/\alpha)}}\right).
\end{split}
\end{equation*}
Therefore, if we let
\[
C_{\varepsilon,2}
=\frac{ \varepsilon^{p+2+2d/\alpha} (c_{**}^{(p)} \Theta_{d,\alpha}^{(p)})^{(p-1)d/\{\alpha(1-(p-1)d/\alpha)\}}}{2(1-(p-1)d/\alpha)}
\exp\left(-2(c_{**}^{(p)} \Theta^{(p)}_{d,\alpha})^{1/(1-(p-1)d/\alpha)}\varepsilon\right),
\]
then, by \eqref{eq:p-iteration-5}, there exists $T>0$ such that for all $t\ge T$,
\begin{equation}\label{eq:p-iteration-7}
\begin{split}
&\int_0^{t-\varepsilon}\frac{(t-s)^{(p+1)+2d/\alpha}}{s^{(p-1)d/\alpha}}
E_{1-(p-1)d/\alpha, 1-(p-1)d/\alpha}
\left(\frac{c_{**}^{(p)} \Theta_{d,\alpha}^{(p)}s^{1-(p-1)d/\alpha}}{t^{(1+d/\alpha)(1-p/2)}}\right)\,\d s\\
&\ge C_{\varepsilon,2}
t^{-(1+d/\alpha)(1-p/2)(p-1)d/\{\alpha(1-(p-1)d/\alpha)\}}
\exp\left(\frac{(c_{**}^{(p)} \Theta_{d,\alpha}^{(p)})^{1/(1-(p-1)d/\alpha)}t}{t^{(1+d/\alpha)(1-p/2)/(1-(p-1)d/\alpha)}}\right).
\end{split}
\end{equation}

By \eqref{eq:p-iteration-0} and \eqref{eq:p-iteration-7},
\begin{equation}\label{eq:}
\begin{split}
&E[|X(t,x)|^p]\\
&\ge \frac{c_{**}^{(p)} C_{\varepsilon}^p\gamma_{d,\alpha}^{(p+1)}}{\kappa_{d,\alpha} t^{(1+d/\alpha)(1-p/2)+p+1+2d/\alpha}}
\Gamma(1-(p-1)d/\alpha) g^{(p)}(t,x)\\
&\quad \times C_{\varepsilon,2}
t^{(1+d/\alpha)(1-p/2)(p-1)d/\{\alpha(1-(p-1)d/\alpha)\}}
\exp\left(\frac{(c_{**}^{(p)} \Theta_{d,\alpha}^{(p)})^{1/(1-(p-1)d/\alpha)}t}{t^{(1+d/\alpha)(1-p/2)/(1-(p-1)d/\alpha)}}\right)\\
&= \frac{c_{**}^{(p)} C_{\varepsilon}^p C_{\varepsilon,2}\gamma_{d,\alpha}^{(p+1)}}{\kappa_{d,\alpha}^2}
\Gamma(1-(p-1)d/\alpha)  t^{-r_0} g(t,x)^{p+1}
\exp\left((c_{**}^{(p)} \Theta_{d,\alpha}^{(p)})^{1/(1-(p-1)d/\alpha)}t^{r_*}  \right)
\end{split}
\end{equation}
with
\[
r_0=
(1+d/\alpha)(1-p/2)\left(1+
\frac{(p-1)d}{\alpha(1-(p-1)d/\alpha)}\right)
+(p+1+d/\alpha).
\]
Moreover, since $\sup_{|x|\ge e^{\eta t^{r_*}}}g(t,x)=g(t,e^{\eta t^{r_*}})$, we get for any $\eta>0$,
\begin{equation}\label{eq:p-iteration-6}
\begin{split}
\sup_{|x|\ge e^{\eta t^{r_*}}}E[|X(t,x)|^p]
&\ge \frac{c_{**}^{(p)} C_{\varepsilon}^p C_{\varepsilon,2}\gamma_{d,\alpha}^{(p+1)}}{\kappa_{d,\alpha}^2}
\Gamma(1-(p-1)d/\alpha)  t^{-r_0} \\
&\quad\times g(t,e^{\eta t^{r_*}})^{p+1}
\exp \left((c_{**}^{(p)} \Theta_{d,\alpha}^{(p)})^{1/(1-(p-1)d/\alpha)}t^{r_*}  \right).
\end{split}
\end{equation}
Since it follows by \eqref{eq:g-def} that
\[
\lim_{t\rightarrow\infty}\frac{1}{t^{r_*}}\log \{g(t,e^{\eta t^{r_*}})^{p+1}\}
=-(p+1)(d+\alpha)\eta,
\]
we have by \eqref{eq:p-iteration-6},
\[
\liminf_{t\rightarrow\infty}\frac{1}{t^{r_*}}\log \left(\sup_{|x|\ge e^{\eta t^{r_*}}}E[|X(t,x)|^p]\right)\\
\ge (c_{**}^{(p)} \Theta_{d,\alpha}^{(p)})^{1/(1-(p-1)d/\alpha)}-\eta(p+1)(d+\alpha).
\]
Hence by letting $\eta_*= (c_{**}^{(p)} \Theta_{d,\alpha}^{(p)})^{1/(1-(p-1)d/\alpha)}/\{(p+1)(d+\alpha)\}$,
the proof is complete.
\end{proof}

\begin{rem}\label{rem:h-moment-2}\rm
Under the setting of Proposition \ref{thm:p<2},
we can show that $E[|X(t,x)|^{1+\alpha/d}]=\infty$ by \eqref{eq:p-iteration-0}
and the similar calculation as in Remark \ref{rem:h-moment-1}.
\end{rem}

\section{Appendix}
\subsection{Auxiliary lemmas}
The following lemma is taken from \cite[Lemmas 5.2 and 5.4]{CK19}.
\begin{lem} \label{lem:mean-bound}
\begin{enumerate}
\item
Let $\{X_{\lambda}\}_{\lambda>0}$ be a Poisson process
such that for each $\lambda>0$,
$X_{\lambda}$ follows the Poisson distribution with parameter $\lambda$.
Then there exists $c>0$ such that for any $r>0$,
\[
E[X_{\lambda}^r] \ge c\times
\begin{dcases}
\lambda, & \quad 0<\lambda<1, \\
\lambda^r, & \quad \lambda\ge 1.\\
\end{dcases}
\]
\item
Let $X$ be a random variable with $E[X]=0$.
Then for any $a\in {\mathbb R}$ and $p\in (1,3]$,
\[
E[|a+X|^p] \ge \kappa^{(p)}(|a|^p+E[|X|^p])
\]
with
\[
\kappa^{(p)}=\begin{dcases}
\frac{1}{4}, &\quad 1<p\le 2, \\
\frac{1}{6}, &\quad 2<p\le 3.
\end{dcases}
\]
\end{enumerate}
\end{lem}

Let $(\Omega,{\mathcal F},\{{\mathcal F}_t\}_{t\ge 0}, P)$ be a filtered probability space,
and let $E$ be a Polish space.
Let $N=N(\d t, \d x)$ be an $\{{\mathcal F}\}_{t\ge 0}$-adapted Poisson random measure on $[0,\infty) \times E$
with intensity measure $m$, and
let
$\tilde{N}=N-m$ be the compensation of $N$
(see, e.g., \cite[Definition 1.20, p.\ 70]{JS03}).
By definition, the measure $m$ is $\sigma$-finite and $m(\{t\}\times E)=0$ for any $t\ge 0$.

\begin{lem}{\rm (\cite[Lemma 3.4]{CK19})}\label{lem:PG-ineq}
Let $H: \Omega \times [0,\infty) \times E \to \R$ be an $\{{\mathcal F}\}_{t\ge 0}$-predictable process
such that the map
\[
t \mapsto \int_{[0,t]\times E} H(s,x)\,\tilde{N}(\d s, \d x)
\]
is a well-defined $\{{\mathcal F}\}_{t\ge 0}$-local martingale.
Then for any $p\in (1,2]$, there exists $c_p>0$,
which is independent of $H$ and $m$, such that
\begin{equation}\label{eq:PG-ineq}
E\left[\left| \int_{[0,\infty)\times E} H(t,x)\,\tilde{N}(\d t, \d x) \right|^p \right]
\ge \frac{c_p \int_{[0,\infty) \times E}E[|H(t,x)|^p]\, m(\d t,\d x)}{\{1\vee m([0,\infty)\times E)\}^{1-p/2}}
\end{equation}
with convention $\infty/\infty:=\infty$.
In particular, the following assertions hold.
\begin{itemize}
\item
If the right hand side of \eqref{eq:PG-ineq} is infinity, then so is the left hand side of \eqref{eq:PG-ineq}.

\item
For any $r_0 \in (1,2]$, $\inf_{p\in [r_0,2]}c_p>0$.
\end{itemize}
\end{lem}

\subsection{Heat kernel: moments and convolutions}\label{section6.2}
In this part, we will generalize \cite[Proposition 3.3, Lemmas 5.2 and 5.3]{CD15-1}
to the
$p$th
power of the function $g(t,x)$ and also to the multidimensional Euclidean space.
For $\nu\in \R$, let $K_{\nu}(x)$ be the modified Bessel function of the second kind, which is defined by
\[
K_{\nu}(x)=\frac{\pi}{2\sin(\nu \pi)}
\left(\sum_{n=0}^{\infty}\frac{(x/2)^{-\nu+2n}}{n!\Gamma(-\nu+n+1)}-\sum_{n=0}^{\infty}\frac{(x/2)^{\nu+2n}}{n!\Gamma(\nu+n+1)}\right), \quad x\ge 0,
\]
 see, e.g., \cite[10.25.2 and 10.27.4]{O25}. In particular, for all $x\ge0$, $K_\nu(x)=K_{-\nu}(x)$.
Let $\langle\cdot,\cdot\rangle$ denote the standard inner product in $\R^d$.

\begin{lem}\label{lem:f-trans}
\begin{enumerate}
\item[{\rm (1)}]
For any $a>0$ and $\mu>(d-2)/2$,
\[
\int_{\R^d}e^{-i\langle x,z\rangle}\frac{1}{(a^2+|x|^2)^{\mu+1}}\,\d x
=\frac{(2\pi)^{d/2}}{2^\mu \Gamma(\mu+1)}\left(\frac{|z|}{a}\right)^{\mu+1-d/2}K_{(d-2)/2-\mu}(a|z|), \quad z\in \R^d.
\]
\item[{\rm (2)}]
For any $p>d/(d+\alpha)$ and $t>0$,
\[
\int_{\R^d}e^{-i\langle x,z\rangle}\frac{1}{(t^{2/\alpha}+|x|^2)^{p(d+\alpha)/2}}\,\d x
\ge c_{\nu}e^{-t^{1/\alpha}|z|}t^{-2\nu/\alpha}, \quad z\in \R^d
\]
with
\begin{equation}\label{eq:const}
\nu=\frac{p(d+\alpha)}{2}-\frac{d}{2}, \quad c_{\nu}=\frac{\pi^{d/2}\Gamma(\nu)}{\Gamma(d/2+\nu)}.
\end{equation}
\end{enumerate}
\end{lem}

\begin{proof} (i)
We first show (1).
Let $J_{\nu}(x)$ be the Bessel function of the first kind (see, e.g., \cite[10.2.2]{O25}).
Then, by \cite[p.\ 577, B.5]{L14} and \cite[p.\ 24, (20)]{EMOT54},
\begin{equation*}
\begin{split}
\int_{\R^d}e^{-i\langle x,z\rangle}\frac{1}{(a^2+|x|^2)^{\mu+1}}\,\d x
&=\frac{(2\pi)^{d/2}}{|z|^{(d-1)/2}}\int_0^{\infty}\frac{r^{(d-1)/2}}{(a^2+r^2)^{\mu+1}}J_{(d-2)/2}(|z|r)\sqrt{|z|r}\,\d r\\
&=\frac{(2\pi)^{d/2}}{2^\mu \Gamma(\mu+1)}\left(\frac{|z|}{a}\right)^{\mu+1-d/2}K_{(d-2)/2-\mu}(a|z|),
\end{split}
\end{equation*}
whence (1) follows.

(ii) We next prove (2). By (1),
\begin{equation}\label{eq:fourier}
\begin{split}
&\int_{\R^d}e^{-i\langle x,z\rangle}\frac{1}{(t^{2/\alpha}+|x|^2)^{p(d+\alpha)/2}}\,\d x\\
&=\frac{(2\pi)^{d/2}}{2^{p(d+\alpha)/2-1} \Gamma(p(d+\alpha)/2)}
\left(\frac{|z|}{t^{1/\alpha}}\right)^{p(d+\alpha)/2-d/2}K_{d/2-p(d+\alpha)/2}(t^{1/\alpha}|z|)\\
&=\frac{(2\pi)^{d/2}}{2^{d/2+\nu-1} \Gamma(d/2+\nu)}
\left(\frac{|z|}{t^{1/\alpha}}\right)^{\nu}K_{-\nu}(t^{1/\alpha}|z|)
=\frac{\pi^{d/2}}{2^{\nu-1} \Gamma(d/2+\nu)}
\left(\frac{|z|}{t^{1/\alpha}}\right)^{\nu}K_{\nu}(t^{1/\alpha}|z|),
\end{split}
\end{equation} where in the second equality we used the definition of $\nu$ in \eqref{eq:const},
and in the last equality we used the relation $K_{-\nu}(x)=K_{\nu}(x)$.

Let $a=t^{1/\alpha}$ and $u=a|z|$.
Then
\begin{equation*}\label{eq:bessel-1}
\left(\frac{|z|}{t^{1/\alpha}}\right)^{\nu}K_{\nu}(t^{1/\alpha}|z|)e^{t^{1/\alpha}|z|}t^{2\nu/\alpha}
=(a|z|)^{\nu} e^{a|z|} K_{\nu}(a|z|)
=u^{\nu}e^u K_{\nu}(u)=:h(u).
\end{equation*}
Since
\[
K_{\nu}(u)\sim
\begin{dcases}
2^{\nu-1}\Gamma(\nu)u^{-\nu}, &\quad u\rightarrow 0, \\
\sqrt{\frac{\pi}{2u}}e^{-u}, &\quad u\rightarrow\infty
\end{dcases}
\]
(see, e.g., \cite[10.30.2 and 10.25.3]{O25}),
we have
\[
h(u)\sim
\begin{dcases}
2^{\nu-1}\Gamma(\nu), & \quad u\rightarrow 0,  \\
\sqrt{\frac{\pi}{2}} u^{\nu-1/2}, &\quad  u\rightarrow\infty.
\end{dcases}
\]
On the other hand, since $\nu K_{\nu}(u)+uK_{\nu}'(u)=-uK_{\nu-1}(u)$ (see, e.g., \cite[10.29.2]{O25}),
we obtain
\[
h'(u)=u^{\nu-1}e^u\{(\nu+u)K_{\nu}(u)+uK_{\nu}'(u)\}=u^{\nu}e^u(K_{\nu}(u)-K_{\nu-1}(u)).
\]
Moreover, we see by the expression
\[
K_{\nu}(u)=\int_0^{\infty}e^{-u\cosh t}\cosh(\nu t)\,\d t
\]
(see, e.g., \cite[10.32.9]{O25})
that $K_{\nu}(u)$ is increasing in $\nu$, and so $h'(u)\ge 0$ for all $u>0$.
Hence $h(u)=u^{\nu}e^u K_{\nu}(u)$ is increasing, and so
\[
\inf_{u>0}h(u)=\lim_{u\rightarrow 0}h(u)=2^{\nu-1}\Gamma(\nu).
\]
Therefore, we get
\[
\left(\frac{|z|}{t^{1/\alpha}}\right)^{\nu}K_{\nu}(t^{1/\alpha}|z|)
\ge 2^{\nu-1}\Gamma(\nu)e^{-t^{1/\alpha}|z|}t^{-2\nu/\alpha}.
\]
Then by \eqref{eq:fourier},
\begin{equation*}
\begin{split}
\int_{\R^d}e^{-i\langle x,z\rangle}\frac{1}{(t^{2/\alpha}+|x|^2)^{p(d+\alpha)/2}}\,\d x
&\ge \frac{\pi^{d/2}}{2^{\nu-1} \Gamma(d/2+\nu)}\cdot 2^{\nu-1}\Gamma(\nu)e^{-t^{1/\alpha}|z|}t^{-2\nu/\alpha}
=c_{\nu}e^{-t^{1/\alpha}|z|}t^{-2\nu/\alpha}.
\end{split}
\end{equation*}
The proof is complete.
\end{proof}

We here recall the following equality related to the Beta function:
for any $p,q>0$,
\begin{equation}\label{eq:beta}
\int_0^{\infty}\frac{y^{p-1}}{(y+1)^{p+q}}\,\d y
=\int_0^1 x^{p-1}(1-x)^{q-1}\,{\rm d}x=B(p,q)=\frac{\Gamma(p)\Gamma(q)}{\Gamma(p+q)}.
\end{equation}
The first equality above follows by the change of variables formula with $y+1=(1-x)^{-1}$.

Define
\begin{equation}\label{eq:g-def}
\begin{split}
\kappa_{d,\alpha}=\frac{\Gamma((d+\alpha)/2)}{\pi^{d/2}\Gamma(\alpha/2)}, \quad
g(t,x)=\frac{\kappa_{d,\alpha}t}{(t^{2/\alpha}+|x|^2)^{(d+\alpha)/2}}, \quad t>0, \ x\in \R^d.
\end{split}
\end{equation}
Then by \eqref{eq:beta}, we have for any $t>0$, $\int_{\R^d}g(t,x)\,\d x=1$.
The next lemma immediately follows from the definition of $g(t,x)$ as well as \eqref{eq:heat} and \eqref{eq:g-asymp}.
\begin{lem}\label{lem:g-comp}
\begin{enumerate}
\item[{\rm (1)}] There exist positive constants $C_{1,g}$ and $C_{2,g}$ with $C_{1,g}<C_{2,g}$ such that
\[
C_{1,g}g(t,x) \le q_t(x)\le C_{2,g}g(t,x), \quad t>0, \, x\in \R^d.
\]
\item[{\rm(2)}] For any $c\ge 1$,
\[
g(t,cx)\ge c^{-(d+\alpha)}g(t,x), \quad t>0, \ x\in \R^d.
\]
\item[{\rm(3)}] The following inequality holds{\rm :}
\[
g(s,x)\ge \frac{s}{t}g(t,x), \quad 0<s\le t, \ x\in \R^d.
\]
\item[{\rm (4)}]
For each fixed $x\in \R^d$, the function $t^{d/\alpha}g(t,x)$ is increasing in $t>0$.
\end{enumerate}
\end{lem}

In the next lemma, we show a  lower bound of the time-space convolution of the functions $g(t,x)^p$.

\begin{lem}\label{lem:g-conv}
\begin{enumerate}
\item[{\rm (1)}]
For any $t>0$ and $x,y\in \R^d$,
\[
g(t,x-y)\ge \kappa_{d,\alpha}^{-1}t^{d/\alpha}g(t,\sqrt{2}x)g(t,\sqrt{2}y).
\]

\item[{\rm (2)}]
Let $p> d/(d+\alpha)$.
Then for any $t>0$ and $z\in \R^d$,
\[
\int_{\R^d}e^{-i\langle x,z\rangle} g(t,x)^p\,\d x
\ge \kappa_{d,\alpha}^p c_{\nu}t^{-(p-1)d/\alpha}e^{-t^{1/\alpha}|z|}.
\]

\item[{\rm (3)}]
Let $p>d/(d+\alpha)$.
Then for any $x\in \R^d$ and $t,s>0$ with $t\ge s>0$,
\[
\{g(t-s,\cdot)^p*g(s,\cdot)^p\}(x)
\ge \gamma_{d,\alpha}^{(p)} \frac{(t-s)^{d/\alpha}}{s^{(p-1)d/\alpha}t^{d/\alpha}}g(t-s,x)^p
\]
with
\[
\gamma_{d,\alpha}^{(p)}=\frac{\kappa_{d,\alpha}^p c_{\nu}^2\omega_d (d-1)!}{2^{p(d+\alpha)+d}(2\pi)^d}.
\]

\item[{\rm (4)}]
For any $x\in \R^d$ and $s,t>0$ with $t\ge s\ge t/2$,
\[
g(s,x)\ge \frac{1}{2^{1+d/\alpha}}\left(\frac{t}{s}\right)^{d/\alpha}g(t,x).
\]

\item[{\rm (5)}]
Let $p\in (d/(d+\alpha),1+\alpha/d)$.
Then for any $n\ge 1$, the following inequality holds for any $t>0$ and $x\in \R^d${\rm :}
\begin{equation}\label{eq:ts-conv}
(\underbrace{g^p\star \dots \star g^p}_{n+1})(t,x)
\ge
\frac{(\Lambda_{d,\alpha}^{(p)})^n \Gamma(1-(p-1)d/\alpha)}{\Gamma((n+1)(1-(p-1)d/\alpha))}t^{n(1-(p-1)d/\alpha)}g(t,x)^p
\end{equation}
with
\[
\Lambda_{d,\alpha}^{(p)}=\frac{\gamma_{d,\alpha}^{(p)}\Gamma(1-(p-1)d/\alpha)}{2^{1+p(1+d/\alpha)}}.
\]

\end{enumerate}
\end{lem}

\begin{proof} (i)
We first prove (1).
Since
\[
1+\left(\frac{|x-y|}{t^{1/\alpha}}\right)^2
\le 1+2\left\{\left(\frac{|x|}{t^{1/\alpha}}\right)^2+\left(\frac{|y|}{t^{1/\alpha}}\right)^2\right\}
\le \left\{1+2\left(\frac{|x|}{t^{1/\alpha}}\right)^2\right\} \left\{1+2\left(\frac{|y|}{t^{1/\alpha}}\right)^2\right\},
\]
we have
\begin{equation*}
\begin{split}
g(t,x-y)
&=\frac{\kappa_{d,\alpha}}{t^{d/\alpha}}\frac{1}{\{1+(|x-y|/t^{1/\alpha})^2\}^{(d+\alpha)/2}}\\
&\ge \frac{\kappa_{d,\alpha}}{t^{d/\alpha}} \frac{1}{\{1+2(|x|/t^{1/\alpha})^2\}^{(d+\alpha)/2}}
\frac{1}{\{1+2(|y|/t^{1/\alpha})^2\}^{(d+\alpha)/2}}\\
&=\frac{\kappa_{d,\alpha}}{t^{d/\alpha}}
\frac{t^{1+d/{\alpha}}}{(t^{2/\alpha}+2|x|^2)^{(d+\alpha)/2}}
\frac{t^{1+d/{\alpha}}}{(t^{2/\alpha}+2|y|^2)^{(d+\alpha)/2}} \\
&=\kappa_{d,\alpha}^{-1}t^{d/\alpha}g(t,\sqrt{2}x)g(t,\sqrt{2}y).
\end{split}
\end{equation*}

(ii) We next prove (2). By Lemma \ref{lem:f-trans}(2),
\begin{equation*}
\begin{split}
\int_{\R^d}e^{-i\langle x,z\rangle} g(t,x)^p\,\d x
&=\kappa_{d,\alpha}^pt^p\int_{\R^d}e^{-i\langle x, z \rangle}\frac{1}{(t^{2/\alpha}+|x|^2)^{p(d+\alpha)/2}}\,\d x\\
&\ge \kappa_{d,\alpha}^pt^p  c_{\nu}e^{-t^{1/\alpha}|z|}t^{-\{p(d+\alpha)-d\}/\alpha}\\
&=\kappa_{d,\alpha}^p c_{\nu}t^{-(p-1)d/\alpha}e^{-t^{1/\alpha}|z|}.
\end{split}
\end{equation*}

(iii) We next prove (3). By (1) and Lemma \ref{lem:g-comp}(2),
\begin{equation}\label{eq:conv-1}
\begin{split}
\{g(t-s,\cdot)^p*g(s,\cdot)^p\}(x)
&=\int_{\R^d} g(t-s,x-y)^p g(s,y)^p\,\d y\\
&\ge \kappa_{d,\alpha}^{-p}(t-s)^{pd/\alpha} g(t-s,\sqrt{2}x)^p \int_{\R^d} g(t-s,\sqrt{2}y)^pg(s,y)^p\,\d y\\
&\ge 2^{-p(d+\alpha)}\kappa_{d,\alpha}^{-p}(t-s)^{pd/\alpha} g(t-s,x)^p \int_{\R^d} g(t-s,y)^pg(s,y)^p\,\d y.
\end{split}
\end{equation}
Then, by the Plancherel theorem and (2),
\begin{align*}
\int_{\R^d} g(t-s,y)^pg(s,y)^p\,\d y
&=\frac{1}{(2\pi)^d}\int_{\R^d} \left(\int_{\R^d}e^{-i\langle x,y\rangle} g(t-s,x)^p\,\d x\right)
\left(\int_{\R^d}e^{-i\langle x,y\rangle} g(s,x)^p\,\d x\right)\,\d y\\
&\ge \frac{\kappa_{d,\alpha}^{2p} c_{\nu}^2}{(2\pi)^d}
\{s(t-s)\}^{-(p-1)d/\alpha}
\int_{\R^d} e^{-\{(t-s)^{1/\alpha}+s^{1/\alpha}\}|y|}\,\d y\\
&=\frac{\kappa_{d,\alpha}^{2p} c_{\nu}^2\omega_d}{(2\pi)^d}
\{s(t-s)\}^{-(p-1)d/\alpha}
\int_0^{\infty} e^{-\{(t-s)^{1/\alpha}+s^{1/\alpha}\}r}r^{d-1}\,\d r\\
&=\frac{\kappa_{d,\alpha}^{2p} c_{\nu}^2\omega_d(d-1)!}{(2\pi)^d}
\frac{\{s(t-s)\}^{-(p-1)d/\alpha}}{\{(t-s)^{1/\alpha}+s^{1/\alpha}\}^d}.
\end{align*}
Hence, by \eqref{eq:conv-1} and the inequality $(t-s)^{1/\alpha}+s^{1/\alpha}\le 2t^{1/\alpha} $ for all $0\le s\le t$,
we obtain
\begin{equation*}
\begin{split}
&\{g(t-s,\cdot)^p*g(s,\cdot)^p\}(x)\\
&\ge 2^{-p(d+\alpha)}\kappa_{d,\alpha}^{-p}(t-s)^{pd/\alpha} g(t-s,x)^p \cdot
\frac{\kappa_{d,\alpha}^{2p} c_{\nu}^2\omega_d(d-1)!}{(2\pi)^d}\frac{\{s(t-s)\}^{-(p-1)d/\alpha}}{\{(t-s)^{1/\alpha}+s^{1/\alpha}\}^d}\\
&=\frac{\kappa_{d,\alpha}^p c_{\nu}^2\omega_d(d-1)!}{2^{p(d+\alpha)}(2\pi)^d}
\frac{(t-s)^{d/\alpha}}{s^{(p-1)d/\alpha}\{(t-s)^{1/\alpha}+s^{1/\alpha}\}^d}g(t-s,x)^p \\
&\ge \frac{\kappa_{d,\alpha}^p c_{\nu}^2\omega_d(d-1)!}{2^{p(d+\alpha)}(2\pi)^d}
\frac{(t-s)^{d/\alpha}}{s^{(p-1)d/\alpha}(2t^{1/\alpha})^d}g(t-s,x)^p\\
&=\gamma_{d,\alpha}^{(p)}\frac{(t-s)^{d/\alpha}}{s^{(p-1)d/\alpha}t^{d/\alpha}}g(t-s,x)^p.
\end{split}
\end{equation*}

(iv) We next prove (4).
For any $x\in \R^d$ and $s,t>0$ with $t\ge s\ge t/2$,
we see by Lemma \ref{lem:g-comp}(4) that
\begin{equation*}
\begin{split}
s^{d/\alpha}g(s,x)
\ge \left(\frac{t}{2}\right)^{d/\alpha}g(t/2,x)
&=\frac{t^{d/\alpha}}{2^{1+d/\alpha}}\frac{\kappa_{d,\alpha} t}{((t/2)^{2/\alpha}+|x|^2)^{(d+\alpha)/2}}\\
&\ge \frac{t^{d/\alpha}}{2^{1+d/\alpha}}\frac{\kappa_{d,\alpha} t}{(t^{2/\alpha}+|x|^2)^{(d+\alpha)/2}}
=\frac{t^{d/\alpha}}{2^{1+d/\alpha}}g(t,x).
\end{split}
\end{equation*}

(v) We finally prove (5) by the induction.
Let us show \eqref{eq:ts-conv} with $n=1$. By (3),
\[
(g^p\star g^p)(t,x)=\int_0^t\{g(t-s,\cdot)^p*g(s,\cdot)^p\}(x)\,\d s
\ge \frac{\gamma_{d,\alpha}^{(p)}}{t^{d/\alpha}}\int_0^t \frac{(t-s)^{d/\alpha}}{s^{(p-1)d/\alpha}}g(t-s,x)^p\,\d s.
\]
Then by (4),
\begin{equation*}
\begin{split}
\int_0^t \frac{(t-s)^{d/\alpha}}{s^{(p-1)d/\alpha}}g(t-s,x)^p\,\d s
&\ge \int_0^{t/2}\frac{(t-s)^{d/\alpha}}{s^{(p-1)d/\alpha}}g(t-s,x)^p\,\d s\\
&\ge \int_0^{t/2}\frac{(t-s)^{d/\alpha}}{s^{(p-1)d/\alpha}}\left\{\frac{1}{2^{1+d/\alpha}}\left(\frac{t}{t-s}\right)^{d/\alpha}g(t,x)\right\}^p\,\d s\\
&= \frac{t^{pd/\alpha}}{2^{p(1+d/\alpha)}} g(t,x)^p
\int_0^{t/2}\frac{1}{(t-s)^{(p-1)d/\alpha}s^{(p-1)d/\alpha}}\,\d s\\
&=\frac{t^{1-2(p-1)d/\alpha+pd/\alpha}}{2^{1+p(1+d/\alpha)}} g(t,x)^p
\frac{\Gamma(1-(p-1)d/\alpha)^2}{\Gamma(2(1-(p-1)d/\alpha))}.
\end{split}
\end{equation*}
The last equality above is a consequence of the following:
\begin{equation}\label{eq:int-beta}
\begin{split}
&\int_0^{t/2}\frac{1}{(t-s)^{(p-1)d/\alpha}s^{(p-1)d/\alpha}}\,\d s\\
&=\frac{1}{2}\int_0^{t/2}\frac{1}{(t-s)^{(p-1)d/\alpha}s^{(p-1)d/\alpha}}\,\d s
+\frac{1}{2}\int_0^{t/2}\frac{1}{(t-s)^{(p-1)d/\alpha}s^{(p-1)d/\alpha}}\,\d s\\
&=\frac{1}{2}\int_0^{t/2}\frac{1}{(t-s)^{(p-1)d/\alpha}s^{(p-1)d/\alpha}}\,\d s
+\frac{1}{2}\int_{t/2}^t\frac{1}{(t-s)^{(p-1)d/\alpha}s^{(p-1)d/\alpha}}\,\d s\\
&=\frac{1}{2}\int_0^t\frac{1}{(t-s)^{(p-1)d/\alpha}s^{(p-1)d/\alpha}}\,\d s\\
&=\frac{1}{2}t^{1-2(p-1)d/\alpha}\frac{\Gamma(1-(p-1)d/\alpha)^2}{\Gamma(2(1-(p-1)d/\alpha))}.
\end{split}
\end{equation}
We here used \eqref{eq:beta} at the last equality above.
Thus
\begin{equation*}
\begin{split}
(g^p\star g^p)(t,x)
&\ge \frac{\gamma_{d,\alpha}^{(p)}}{2^{1+p(1+d/\alpha)}}\frac{\Gamma(1-(p-1)d/\alpha)^2}{\Gamma(2(1-(p-1)d/\alpha))}
t^{1-(p-1)d/\alpha}
g(t,x)^p\\
&=\frac{\Lambda_{d,\alpha}^{(p)}\Gamma(1-(p-1)d/\alpha)}{\Gamma(2(1-(p-1)d/\alpha))}
t^{1-(p-1)d/\alpha}g(t,x)^p,
\end{split}
\end{equation*}
which is \eqref{eq:ts-conv} with $n=1$.

We now suppose that \eqref{eq:ts-conv} holds with some $n\ge 1$.
Then
\begin{equation}\label{eq:ts-conv-1}
\begin{split}
(\underbrace{g^p\star \dots \star g^p}_{n+2})(t,x)
&=\int_0^t\left(\int_{\R^d}(\underbrace{g^p\star \dots \star g^p}_{n+1})(t-s,x-y)g(s,y)^p\,\d y\right)\,\d s\\
&\ge
\frac{(\Lambda_{d,\alpha}^{(p)})^n\Gamma(1-(p-1)d/\alpha)}{\Gamma((n+1)(1-(p-1)d/\alpha))} \\
&\quad\times \int_0^t (t-s)^{n(1-(p-1)d/\alpha)} \left(\int_{\R^d}g(t-s,x-y)^pg(s,y)^p\,\d y\right)\,\d s\\
&=\frac{(\Lambda_{d,\alpha}^{(p)})^n \Gamma(1-(p-1)d/\alpha)}{\Gamma((n+1)(1-(p-1)d/\alpha))} \\
&\quad\times \int_0^t (t-s)^{n(1-(p-1)d/\alpha)}\{g(t-s, \cdot)^p*g(s,\cdot)^p\}(x)\,\d s.
\end{split}
\end{equation}
By (3) and (4), we see that if $0\le s\le t/2$, then
\begin{equation*}
\begin{split}
\{g(t-s, \cdot)^p*g(s,\cdot)^p\}(x)
&\ge \gamma_{d,\alpha}^{(p)} \frac{(t-s)^{d/\alpha}}{s^{(p-1)d/\alpha}t^{d/\alpha}}g(t-s,x)^p\\
&\ge \gamma_{d,\alpha}^{(p)} \frac{(t-s)^{d/\alpha}}{s^{(p-1)d/\alpha}t^{d/\alpha}}
\left\{\frac{1}{2^{1+d/\alpha}}\left(\frac{t}{t-s}\right)^{d/\alpha}g(t,x)\right\}^p\\
&=\frac{\gamma_{d,\alpha}^{(p)}}{2^{p(1+d/\alpha)}}\frac{t^{(p-1)d/\alpha}}{s^{(p-1)d/\alpha}(t-s)^{(p-1)d/\alpha}}g(t,x)^p.
\end{split}
\end{equation*}
Therefore,
\begin{equation*}
\begin{split}
&\int_0^t (t-s)^{n(1-(p-1)d/\alpha)} \{g(t-s,\cdot)^p*g(s,\cdot)^p\}(x)\,\d s\\
&\ge  \frac{\gamma_{d,\alpha}^{(p)}}{2^{p(1+d/\alpha)}}
\int_0^{t/2} (t-s)^{n(1-(p-1)d/\alpha)} \frac{t^{(p-1)d/\alpha}}{s^{(p-1)d/\alpha}(t-s)^{(p-1)d/\alpha}}\,\d s \cdot  g(t,x)^p.
\end{split}
\end{equation*}
Combining this with \eqref{eq:ts-conv-1}, we obtain
\begin{align*}
(\underbrace{g^p\star \dots \star g^p}_{n+2})(t,x)
&\ge
\frac{\gamma_{d,\alpha}^{(p)}}{2^{p(1+d/\alpha)}}
\frac{(\Lambda_{d,\alpha}^{(p)})^n\Gamma(1-(p-1)d/\alpha)}{\Gamma((n+1)(1-(p-1)d/\alpha))} \\
&\times\int_0^{t/2} (t-s)^{n(1-(p-1)d/\alpha)} \frac{t^{(p-1)d/\alpha}}{s^{(p-1)d/\alpha}(t-s)^{(p-1)d/\alpha}}\,\d s \cdot  g(t,x)^p.
\end{align*}
In the same way, we also obtain
\begin{equation*}
\begin{split}
(\underbrace{g^p\star \dots \star g^p}_{n+2})(t,x)
&=\int_0^t\left(\int_{\R^d}(\underbrace{g^p\star \dots \star g^p}_{n+1})(s,x-y)g(t-s,y)^p\,\d y\right)\,\d s\\
&\ge
\frac{\gamma_{d,\alpha}^{(p)}}{2^{p(1+d/\alpha)}} \frac{(\Lambda_{d,\alpha}^{(p)})^n\Gamma(1-(p-1)d/\alpha)}{\Gamma((n+1)(1-(p-1)d/\alpha))} \\
&\quad\times\int_0^{t/2} s^{n(1-(p-1)d/\alpha)} \frac{t^{(p-1)d/\alpha}}{s^{(p-1)d/\alpha}(t-s)^{(p-1)d/\alpha}}\,\d s \cdot  g(t,x)^p\\
&=\frac{\gamma_{d,\alpha}^{(p)}}{2^{p(1+d/\alpha)}}
\frac{(\Lambda_{d,\alpha}^{(p)})^n\Gamma(1-(p-1)d/\alpha)}{\Gamma((n+1)(1-(p-1)d/\alpha))} \\
&\quad\times\int_{t/2}^t (t-s)^{n(1-(p-1)d/\alpha)} \frac{t^{(p-1)d/\alpha}}{s^{(p-1)d/\alpha}(t-s)^{(p-1)d/\alpha}}\,\d s \cdot  g(t,x)^p.
\end{split}
\end{equation*}
By the two inequalities above with \eqref{eq:beta}, we get
\begin{align*}
&(\underbrace{g^p\star \dots \star g^p}_{n+2})(t,x)\\
&\ge \frac{1}{2}\cdot \frac{\gamma_{d,\alpha}^{(p)}}{2^{p(1+d/\alpha)}}
\frac{(\Lambda_{d,\alpha}^{(p)})^n\Gamma(1-(p-1)d/\alpha)}{\Gamma((n+1)(1-(p-1)d/\alpha))} \\
&\quad \times\int_0^t (t-s)^{n(1-(p-1)d/\alpha)} \frac{t^{(p-1)d/\alpha}}{s^{(p-1)d/\alpha}(t-s)^{(p-1)d/\alpha}}\,\d s \cdot  g(t,x)^p\\
&=\frac{(\Lambda_{d,\alpha}^{(p)})^{n+1}}{\Gamma((n+1)(1-(p-1)d/\alpha))}\\
&\quad \times \frac{\Gamma(1-(p-1)d/\alpha)\Gamma((n+1)(1-(p-1)d/\alpha))}{\Gamma((n+2)(1-(p-1)d/\alpha))}t^{(n+1)(1-(p-1)d/\alpha)}g(t,x)^p\\
&=\frac{(\Lambda_{d,\alpha}^{(p)})^{n+1}\Gamma(1-(p-1)d/\alpha)}{\Gamma((n+2)(1-(p-1)d/\alpha))}t^{(n+1)(1-(p-1)d/\alpha)} g(t,x)^p.
\end{align*}
Hence the induction is complete and so (5) follows.
\end{proof}

For $a,b>0$, let $E_{a,b}(z)$ be the two-parameter Mittag-Leffler function defined by
\begin{equation}\label{eq:ML}
E_{a,b}(z)=\sum_{n=0}^{\infty}\frac{z^n}{\Gamma(an+b)}, \quad z\in \R.
\end{equation}
We know by \cite[Theorem 4.3]{GKMR14} that, if $a\in (0,2)$ and $b>0$, then
\begin{equation}\label{eq:ML-1}
E_{a,b}(z)\sim \frac{1}{a}z^{(1-b)/a}\exp(z^{1/a}), \quad z\rightarrow \infty.
\end{equation}

For $p>d/(d+\alpha)$ and $c>0$, we define
\begin{equation}\label{eq:k-def}
K^{(p)}(c;t,x)=\sum_{n=0}^{\infty}c^n(\underbrace{g^p\star \cdots\star g^p}_{n+1})(t,x).
\end{equation}
\begin{lem}\label{prop:k-lower}
Let $p\in (d/(d+\alpha),1+\alpha/d)$ and $c>0$. Then for any $t>0$ and $x\in \R^d$,
\[
K^{(p)}(c;t,x)
\ge \Gamma(1-(p-1)d/\alpha) g(t,x)^p E_{1-(p-1)d/\alpha, 1-(p-1)d/\alpha}
\left(c\Lambda_{d,\alpha}^{(p)} t^{1-(p-1)d/\alpha}\right).
\]
\end{lem}
\begin{proof}
By Lemma \ref{lem:g-conv}(5),
\begin{equation*}
\begin{split}
K^{(p)}(c;t,x)
&\ge\Gamma(1-(p-1)d/\alpha) g(t,x)^p
\sum_{n=0}^{\infty}\frac{\left(c \Lambda_{d,\alpha}^{(p)} t^{1-(p-1)d/\alpha}\right)^n}{\Gamma((n+1)(1-(p-1)d/\alpha))}\\
&=
\Gamma(1-(p-1)d/\alpha)
g(t,x)^pE_{1-(p-1)d/\alpha, 1-(p-1)d/\alpha}
\left(c \Lambda_{d,\alpha}^{(p)} t^{1-(p-1)d/\alpha}\right).
\end{split}
\end{equation*}
The proof is complete.
\end{proof}

\begin{lem}
 \begin{enumerate}
\item[{\rm (1)}]
Let $p>d/(d+\alpha)$. Then for
any $t>0$,
\[
\int_{\R^d}g(t,y)^p\,\d y
=\frac{\kappa_{d,\alpha}^p\pi^{d/2}}{t^{(p-1)d/\alpha}}\frac{\Gamma(p(d+\alpha)/2-d/2)}{\Gamma(p(d+\alpha)/2)}.
\]
In particular,
\begin{equation}\label{eq:g-moment}
\int_{\R^d}g(t,y)^2\,\d y
=\frac{\kappa_{d,\alpha}^2\pi^{d/2}}{t^{d/\alpha}}\frac{\Gamma(d/2+\alpha)}{\Gamma(d+\alpha)}.
\end{equation}

\item[{\rm (2)}]
Let $p\in (d/(d+\alpha), 1+\alpha/d)$. Then for
any $t>0$ and $x\in \R^d$,
\[
(1\star g^p)(t,x)
=\frac{\kappa_{d,\alpha}^p \pi^{d/2}}{1-(p-1)d/\alpha}
\frac{\Gamma(p(d+\alpha)/2-d/2)}{\Gamma(p(d+\alpha)/2)}
t^{1-(p-1)d/\alpha},
\quad t>0, \, x\in \R^d.
\]
\end{enumerate}
\end{lem}

\begin{proof} (i)
We first prove (1).
For any $p>d/(d+\alpha)$,
we have by the change of variables formula with $r=t^{1/\alpha}\sqrt{u}$,
and by \eqref{eq:beta},
\begin{equation*}
\begin{split}
\int_{\R^d}\frac{1}{(t^{2/\alpha}+|y|^2)^{p(d+\alpha)/2}}\,\d y
&=\omega_d \int_0^{\infty}\frac{r^{d-1}}{(t^{2/\alpha}+r^2)^{p(d+\alpha)/2}}\,\d r\\
&=\frac{\omega_d}{2t^{p+(p-1)d/\alpha}} \int_0^{\infty}\frac{u^{d/2-1}}{(1+u)^{p(d+\alpha)/2}}\,\d u\\
&=\frac{\omega_d}{2t^{p+(p-1)d/\alpha}}\frac{\Gamma(d/2)\Gamma(p(d+\alpha)/2-d/2)}{\Gamma(p(d+\alpha)/2)}\\
&=\frac{\pi^{d/2}}{t^{p+(p-1)d/\alpha}}\frac{\Gamma(p(d+\alpha)/2-d/2)}{\Gamma(p(d+\alpha)/2)}.
\end{split}
\end{equation*}
At the last equality above, we also used the relation $\omega_d=2\pi^{d/2}/\Gamma(d/2)$.
Hence by the definition of $g(t,x)$, we have
\[
\int_{\R^d}g(t,y)^p\,\d y
=\kappa_{d,\alpha}^p t^p\int_{\R^d}\frac{1}{(t^{2/\alpha}+|y|^2)^{p(d+\alpha)/2}}\,\d y
=\frac{\kappa_{d,\alpha}^p\pi^{d/2}}{t^{(p-1)d/\alpha}}\frac{\Gamma(p(d+\alpha)/2-d/2)}{\Gamma(p(d+\alpha)/2)}.
\]
The proof of (1) is complete.

(ii)
Let $p\in (d/(d+\alpha),1+\alpha/d)$.
Since (1) yields
\begin{equation*}
\begin{split}
(1\star g^p)(t,x)
=\int_0^t \left(\int_{\R^d}g(s,y)^p\,\d y\right)\,\d s
&= \kappa_{d,\alpha}^p\pi^{d/2} \frac{\Gamma(p(d+\alpha)/2-d/2)}{\Gamma(p(d+\alpha)/2)}
\int_0^t s^{-(p-1)d/\alpha}\,\d s\\
&= \frac{\kappa_{d,\alpha}^p \pi^{d/2}}{1-(p-1)d/\alpha}
\frac{\Gamma(p(d+\alpha)/2-d/2)}{\Gamma(p(d+\alpha)/2)} t^{1-(p-1)d/\alpha},
\end{split}
\end{equation*}
we have (2).
\end{proof}

To consider the case for $p\in (1,2)$, we need the following two lemmas, which are variation versions of Lemma \ref{lem:g-conv}(5) and Lemma \ref{prop:k-lower}, respectively.
For $p>d/(d+\alpha)$ and $(t,x)\in (0,\infty)\times \R^d$,
let $g^{(p)}(t,x)=g(t,x)^{p+1}/g(t,0)$.
For $c>0$, we define
\begin{equation}\label{eq:k1-def}
K_1^{(p)}(c;t,x)=\sum_{n=0}^{\infty}c^n(\underbrace{g^{(p)}\star \cdots\star g^{(p)}}_{n+1})(t,x), \quad t>0, \ x\in \R^d.
\end{equation}

\begin{lem}\label{lem:p-g-conv}
Let $p\in (d/(d+\alpha),1+\alpha/d)$.
Then for any $n\ge 1$, the following inequality holds for any $t>0$ and $x\in \R^d${\rm :}
\begin{equation}\label{eq:p-ts-conv}
(\underbrace{g^{(p)}\star \dots \star g^{(p)}}_{n+1})(t,x)
\ge
\frac{(\Theta_{d,\alpha}^{(p)})^n \Gamma(1-(p-1)d/\alpha)}{\Gamma((n+1)(1-(p-1)d/\alpha))}t^{n(1-(p-1)d/\alpha)}g^{(p)}(t,x)
\end{equation}
with
\[
\Theta_{d,\alpha}^{(p)}=\frac{\gamma_{d,\alpha}^{(p+1)}\Gamma(1-(p-1)d/\alpha)}{2^{1+(p+1)(1+d/\alpha)}\kappa_{d,\alpha}}.
\]
\end{lem}

\begin{proof}
Following the proof of \eqref{eq:ts-conv}, we
prove \eqref{eq:p-ts-conv} by induction.
We first show \eqref{eq:p-ts-conv} with $n=1$. By Lemma \ref{lem:g-conv}(3),
\begin{equation*}
\begin{split}
(g^{(p)}\star g^{(p)})(t,x)
&=\int_0^t \left(\int_{\R^d} g^{(p)}(t-s,x-y)g^{(p)}(s,y) \,\d y\right)\,\d s\\
&=\int_0^t \frac{1}{g(t-s,0)g(s,0)}\left(\int_{\R^d} g(t-s,x-y)^{p+1} g(s,y)^{p+1} \,\d y\right)\,\d s\\
&\ge \frac{\gamma_{d,\alpha}^{(p+1)}}{t^{d/\alpha}}\int_0^t \frac{1}{g(t-s,0)g(s,0)} \frac{(t-s)^{d/\alpha}}{s^{pd/\alpha}}g(t-s,x)^{p+1}\,\d s\\
&=\frac{\gamma_{d,\alpha}^{(p+1)} }{\kappa_{d,\alpha}^2 t^{d/\alpha}}\int_0^t  \frac{(t-s)^{2d/\alpha}}{s^{(p-1)d/\alpha}}g(t-s,x)^{p+1}\,\d s.
\end{split}
\end{equation*}
Then by Lemma \ref{lem:g-conv}(4) and \eqref{eq:int-beta},
\begin{equation*}
\begin{split}
\int_0^t \frac{(t-s)^{2d/\alpha}}{s^{(p-1)d/\alpha}}g(t-s,x)^{p+1}\,\d s
&\ge \int_0^{t/2}\frac{(t-s)^{2d/\alpha}}{s^{(p-1)d/\alpha}}g(t-s,x)^{p+1}\,\d s\\
&\ge \int_0^{t/2}\frac{(t-s)^{2d/\alpha}}{s^{(p-1)d/\alpha}}\left\{\frac{1}{2^{1+d/\alpha}}\left(\frac{t}{t-s}\right)^{d/\alpha}g(t,x)\right\}^{p+1}\,\d s\\
&= \frac{t^{(p+1)d/\alpha}}{2^{(p+1)(1+d/\alpha)}} g(t,x)^{p+1}
\int_0^{t/2}\frac{1}{(t-s)^{(p-1)d/\alpha}s^{(p-1)d/\alpha}}\,\d s\\
&=\frac{t^{1-(p-3)d/\alpha}}{2^{1+(p+1)(1+d/\alpha)}}g(t,x)^{p+1}
\frac{\Gamma(1-(p-1)d/\alpha)^2}{\Gamma(2(1-(p-1)d/\alpha))}\\
&=\frac{\kappa_{d,\alpha} }{2^{1+(p+1)(1+d/\alpha)}}
\frac{\Gamma(1-(p-1)d/\alpha)^2}{\Gamma(2(1-(p-1)d/\alpha))}
t^{1-(p-2)d/\alpha}g^{(p)}(t,x).
\end{split}
\end{equation*}
Thus
\begin{equation*}
\begin{split}
(g^{(p)}\star g^{(p)})(t,x)
&\ge \frac{\gamma_{d,\alpha}^{(p+1)}}{ 2^{1+(p+1)(1+d/\alpha)}\kappa_{d,\alpha}}\frac{\Gamma(1-(p-1)d/\alpha)^2}{\Gamma(2(1-(p-1)d/\alpha))}
t^{1-(p-1)d/\alpha}g^{(p)}(t,x)\\
&=\frac{\Theta_{d,\alpha}^{(p)} \Gamma(1-(p-1)d/\alpha)}{\Gamma(2(1-(p-1)d/\alpha))}
t^{1-(p-1)d/\alpha}g^{(p)}(t,x),
\end{split}
\end{equation*}
which is \eqref{eq:p-ts-conv} with $n=1$.

We now suppose that \eqref{eq:p-ts-conv} holds with some $n\ge 1$.
Then
\begin{equation}\label{eq:p-ts-conv-1}
\begin{split}
&(\underbrace{g^{(p)}\star \dots \star g^{(p)}}_{n+2})(t,x)\\
&=\int_0^t\left(\int_{\R^d}(\underbrace{g^{(p)} \star \dots \star g^{(p)}}_{n+1})(t-s,x-y) g^{(p)}(s,y)\,\d y\right)\,\d s\\
&\ge\frac{(\Theta_{d,\alpha}^{(p)})^n\Gamma(1-(p-1)d/\alpha)}{\Gamma((n+1)(1-(p-1)d/\alpha))}
\int_0^t (t-s)^{n(1-(p-1)d/\alpha)}\{g^{(p)}(t-s, \cdot)*g^{(p)}(s,\cdot)\}(x)\,\d s.
\end{split}
\end{equation}
By Lemma \ref{lem:g-conv}(3) and (4), we see that if $0\le s\le t/2$, then
\begin{equation*}
\begin{split}
\{g^{(p)}(t-s, \cdot)*g^{(p)}(s,\cdot)\}(x)
&\ge \frac{\gamma_{d,\alpha}^{(p+1)}}{\kappa_{d,\alpha}^2}\frac{(t-s)^{2d/\alpha}}{s^{(p-1)d/\alpha}t^{d/\alpha}}g(t-s,x)^{p+1}\\
&\ge \frac{\gamma_{d,\alpha}^{(p+1)}}{\kappa_{d,\alpha}^2}\frac{(t-s)^{2d/\alpha}}{s^{(p-1)d/\alpha}t^{d/\alpha}}
\left\{\frac{1}{2^{1+d/\alpha}}\left(\frac{t}{t-s}\right)^{d/\alpha}g(t,x)\right\}^{p+1}\\
&=\frac{\gamma_{d,\alpha}^{(p+1)}}{2^{(p+1)(1+d/\alpha)}\kappa_{d,\alpha}^2}\frac{t^{pd/\alpha}}{s^{(p-1)d/\alpha}(t-s)^{(p-1)d/\alpha}}g(t,x)^{p+1}.
\end{split}
\end{equation*}
Therefore,
\begin{equation*}
\begin{split}
&\int_0^t (t-s)^{n(1-(p-1)d/\alpha)} \{g(t-s,\cdot)^p*g(s,\cdot)^p\}(x)\,\d s\\
&\ge \frac{\gamma_{d,\alpha}^{(p+1)}}{2^{(p+1)(1+d/\alpha)}\kappa_{d,\alpha}^2}
\int_0^{t/2} (t-s)^{n(1-(p-1)d/\alpha)} \frac{t^{pd/\alpha}}{s^{(p-1)d/\alpha}(t-s)^{(p-1)d/\alpha}}\,\d s \cdot  g(t,x)^{p+1}.
\end{split}
\end{equation*}
Combining this with \eqref{eq:p-ts-conv-1}, we obtain
\begin{align*}
(\underbrace{g^{(p)}\star \dots \star g^{(p)}}_{n+2})(t,x)
&\ge
\frac{\gamma_{d,\alpha}^{(p+1)}}{2^{(p+1)(1+d/\alpha)}\kappa_{d,\alpha}^2} \frac{(\Theta_{d,\alpha}^{(p)})^n\Gamma(1-(p-1)d/\alpha)}{\Gamma((n+1)(1-(p-1)d/\alpha))} \\
&\quad \times \int_0^{t/2} (t-s)^{n(1-(p-1)d/\alpha)} \frac{t^{pd/\alpha}}{s^{(p-1)d/\alpha}(t-s)^{(p-1)d/\alpha}}\,\d s \cdot  g(t,x)^{p+1}.
\end{align*}
In the same way, we also obtain
\begin{equation*}
\begin{split}
(\underbrace{g^{(p)}\star \dots \star g^{(p)}}_{n+2})(t,x)
&=\int_0^t\left(\int_{\R^d}(\underbrace{g^{(p)}\star \dots \star g^{(p)}}_{n+1})(s,x-y)g^{(p)}(t-s,y)\,\d y\right)\,\d s\\
&\ge
\frac{\gamma_{d,\alpha}^{(p+1)}}{2^{(p+1)(1+d/\alpha)}\kappa_{d,\alpha}^2} \frac{(\Theta_{d,\alpha}^{(p)})^n\Gamma(1-(p-1)d/\alpha)}{\Gamma((n+1)(1-(p-1)d/\alpha))} \\
&\quad\times  \int_0^{t/2} s^{n(1-(p-1)d/\alpha)} \frac{t^{pd/\alpha}}{s^{(p-1)d/\alpha}(t-s)^{(p-1)d/\alpha}}\,\d s \cdot  g(t,x)^{p+1}\\
&=\frac{\gamma_{d,\alpha}^{(p+1)}}{2^{(p+1)(1+d/\alpha)}\kappa_{d,\alpha}^2} \frac{(\Theta_{d,\alpha}^{(p)})^n\Gamma(1-(p-1)d/\alpha)}{\Gamma((n+1)(1-(p-1)d/\alpha))} \\
&\quad\times \int_{t/2}^t (t-s)^{n(1-(p-1)d/\alpha)} \frac{t^{pd/\alpha}}{s^{(p-1)d/\alpha}(t-s)^{(p-1)d/\alpha}}\,\d s \cdot  g(t,x)^{p+1}.
\end{split}
\end{equation*}
By the two inequalities above with \eqref{eq:beta}, we get
\begin{equation*}
\begin{split}
&(\underbrace{g^{(p)}\star \dots \star g^{(p)}}_{n+2})(t,x)\\
&\ge \frac{1}{2}\cdot \frac{\gamma_{d,\alpha}^{(p+1)}}{2^{(p+1)(1+d/\alpha)}\kappa_{d,\alpha}^2}
\frac{(\Theta_{d,\alpha}^{(p)})^n\Gamma(1-(p-1)d/\alpha)}{\Gamma((n+1)(1-(p-1)d/\alpha))}\\
&\quad \times \int_0^t (t-s)^{n(1-(p-1)d/\alpha)} \frac{t^{pd/\alpha}}{s^{(p-1)d/\alpha}(t-s)^{(p-1)d/\alpha}}\,\d s \cdot  g(t,x)^{p+1}\\
&=\frac{(\Theta_{d,\alpha}^{(p)})^{n+1}}{\kappa_{d,\alpha}\Gamma((n+1)(1-(p-1)d/\alpha))}\\
&\quad \times \frac{\Gamma(1-(p-1)d/\alpha)\Gamma((n+1)(1-(p-1)d/\alpha))}{\Gamma((n+2)(1-(p-1)d/\alpha))}t^{(n+1)(1-(p-1)d/\alpha)+d/\alpha}g(t,x)^{p+1}\\
&=\frac{(\Theta_{d,\alpha}^{(p)})^{n+1}\Gamma(1-(p-1)d/\alpha)}{\Gamma((n+2)(1-(p-1)d/\alpha))}
t^{(n+1)(1-(p-1)d/\alpha)} \frac{t^{d/\alpha}}{\kappa_{d,\alpha}}g(t,x)^{p+1}\\
&=\frac{(\Theta_{d,\alpha}^{(p)})^{n+1}\Gamma(1-(p-1)d/\alpha)}{\Gamma((n+2)(1-(p-1)d/\alpha))}t^{(n+1)(1-(p-1)d/\alpha)}g^{(p)}(t,x).
\end{split}
\end{equation*}
Hence the induction is complete and so \eqref{eq:p-ts-conv} follows.
\end{proof}

\begin{lem}\label{prop:p-k-lower}
Let $p\in (d/(d+\alpha),1+\alpha/d)$ and $c>0$. Then for any $t>0$ and $x\in \R^d$,
\[
K_1^{(p)}(c;t,x)
\ge
\Gamma(1-(p-1)d/\alpha)
g^{(p)}(t,x) E_{1-(p-1)d/\alpha, 1-(p-1)d/\alpha}
\left(c \Theta_{d,\alpha}^{(p)}t^{1-(p-1)d/\alpha}\right).
\]
\end{lem}
\begin{proof}
By Lemma \ref{lem:p-g-conv} and \eqref{eq:ML}
\begin{equation*}
\begin{split}
K_1^{(p)}(c;t,x)
&\ge\Gamma(1-(p-1)d/\alpha) g^{(p)}(t,x)
\sum_{n=0}^{\infty}\frac{\left(c \Theta_{d,\alpha}^{(p)} t^{1-(p-1)d/\alpha}\right)^n}{\Gamma((n+1)(1-(p-1)d/\alpha))}\\
&=
\Gamma(1-(p-1)d/\alpha)
g^{(p)}(t,x) E_{1-(p-1)d/\alpha, 1-(p-1)d/\alpha}
\left(c \Theta_{d,\alpha}^{(p)}t^{1-(p-1)d/\alpha}\right).
\end{split}
\end{equation*}
The proof is complete.
\end{proof}

\noindent
{\bf Acknowledgements.}
The research of Yuichi Shiozawa is supported by JSPS KAKENHI Grant Numbers JP22K18675 and JP23K25773.
The research of Jian Wang is supported by the NSF of China the National Key R\&D Program of China (2022YFA1006003) and the National Natural Science
Foundation of China (Nos. 12225104 and 12531007).

\end{document}